\documentclass[11pt]{amsart}
\usepackage[top=1in, bottom=1in, left=1in, right=1in]{geometry}
\usepackage{amsfonts}
\usepackage{enumerate}
\usepackage{amsmath}
\usepackage{amssymb}
\usepackage{bbm}
\usepackage{placeins}
\numberwithin{equation}{section}
\usepackage{epsfig}
\usepackage{graphicx}
\usepackage{tabularx}
\usepackage[usenames,dvipsnames]{xcolor}
\usepackage{tikz}
\usepackage{times}
\usepackage{comment}
\usepackage{mathtools}
\usepackage{float}
\usepackage{bbm}
\usepackage{bm}
\usepackage{esvect}
\usepackage[backend=biber, sorting=nyt, url=false,
maxnames = 100,
doi=false]{biblatex}
\usepackage{hyperref}
\usepackage{breakurl}
\hypersetup{
colorlinks = true,
linkcolor={black},
urlcolor={blue},
citecolor={blue},    
urlcolor = {blue},
citebordercolor = {0.33 .58 0.33},
linkbordercolor = {0.99 .28 0.23}}

\newcommand{\kommentar}[1]{}

\newcommand{\mb}{\mathbb}

\newcommand{\mf}{\mathfrak}
\newcommand{\ol}{\overline}

\newcommand{\n}{\mf{n}_D}

\newcommand{\m}{\mathrm{m}}

\newtheorem{thm}{Theorem}[section]
\newtheorem{defn}[thm]{Definition}
\newtheorem{prop}[thm]{Proposition}
\newtheorem{coro}[thm]{Corollary}
\newtheorem{lem}[thm]{Lemma}
\newtheorem*{hypothesis*}{Hypothesis}
\newtheorem{conj}[thm]{Conjecture}

\theoremstyle{remark}
\newtheorem{rem}[thm]{Remark}
\DeclareMathOperator{\SL}{SL}
\makeatletter
\let\c@equation\c@thm
\makeatother
\title[An analogue of the Dedekind Eta function for 
Hecke Groups $H(\sqrt{D})$]{\boldmath An analogue of the Dedekind Eta function \\
for 
Hecke Groups $H(\sqrt{D})$}
\thanks{Dorian Goldfeld is partially supported by grant number 567168 from the Simons Foundation.}
\author[Basak]{Debmalya Basak}
\address{
Debmalya Basak: Department of Mathematics,
University of Illinois Urbana-Champaign,
Urbana, IL, 61801, USA and Max Planck Institute for Mathematics, Vivatsgasse 7, 53111 Bonn
Germany}
\email{dbasak2@illinois.edu, basakd@mpim-bonn.mpg.de}

\author[Goldfeld]{Dorian Goldfeld}
\address{
Dorian Goldfeld: Department of Mathematics, Columbia University, 2990 Broadway, New York, NY 10027, USA}
\email{goldfeld@columbia.edu}

\author[Heap]{Winston Heap}
\address{
Winston Heap: Max Planck Institute for Mathematics, Vivatsgasse 7, 53111 Bonn
Germany}
\email{winstonheap@gmail.com}

\author[Robles]{Nicolas Robles}
\address{Nicolas Robles: RAND Corporation, Engineering and Applied Sciences, 1200 S Hayes, Arlington, VA 22202, USA}
\email{nrobles@rand.org}

\author[Zaharescu]{Alexandru Zaharescu}
\address{
Alexandru Zaharescu: Department of Mathematics,
University of Illinois Urbana-Champaign,
Altgeld Hall, 1409 W. Green Street,
Urbana, IL, 61801, USA and Simion Stoilow Institute of Mathematics of the Romanian Academy, 
P. O. Box 1-764, RO-014700 Bucharest, Romania}
\email{zaharesc@illinois.edu} 
\begin{document}


\maketitle
\begin{abstract}
Let $D\equiv 1\bmod{4}$ be a fundamental discriminant of a real quadratic field. We construct an analogue of the classical Dedekind eta function for the Hecke group $H(\sqrt{D})$. This gives rise to a new family of holomorphic modular functions for $H(\sqrt{D})$ which vanish at
the cusp at $\infty$. We establish results on the asymptotic growth and sign patterns of the Fourier coefficients associated to these modular forms.
\end{abstract}

\setcounter{tocdepth}{1}
\section{Introduction}\label{sec: Introduction} 
Let $\lambda>0$. Hecke \cite{Hecke} introduced the matrix groups $H(\lambda)$ generated by the two matrices
$$T_\lambda:= \begin{pmatrix} 1&\lambda\\0&1\end{pmatrix},\quad S = \begin{pmatrix} 0&-1\\1&\phantom{-}0\end{pmatrix}.$$ These are now called Hecke groups.
In the case $\lambda = 1$, the Hecke group $H(\lambda)$ is just the modular group $\SL(2,\mathbb Z).$  The Hecke groups act on the upper half plane $$\mathfrak h := \{x+iy\mid x\in\mathbb R,\; y>0\}$$ in the usual manner. If $\gamma=\left(\begin{smallmatrix} a&b\\c&d \end{smallmatrix}\right)\in H(\lambda)$ and $z\in\mathfrak h$, the action of $\gamma$ on $z$ is denoted $\gamma z$ and is given by $$\gamma z = \frac{az+b}{cz+d}.$$ 

\par
Hecke proved in \cite{Hecke}  that if 
$$\lambda = 2\cos(\pi/q),\;\qquad  (\text{\rm for an integer}\; q >3\; \text{\rm or} \; q=\infty),$$
then $H(\lambda)$ is a discrete subgroup of $\SL(2,\mathbb R)$ with fundamental domain $H(\lambda)\backslash \mathfrak h$ having finite area, that is, $H(\lambda)$ is a Fuchsian group of the first kind.  He also proved that if $\lambda$ is any real number satisfying
$\lambda > 2$ then $H(\lambda)$ is a discrete subgroup of $\SL(2,\mathbb R)$ with fundamental domain $H(\lambda)\backslash \mathfrak h$ having infinite area, that is, $H(\lambda)$ is a Fuchsian group of the second kind. Moreover, he showed that these are the only values of $\lambda$ for which $H(\lambda)$ is a discrete subgroup of $\SL(2,\mathbb R)$. In the case when $\lambda > 2$, the fundamental domain is given in Figure \ref{fig:fundamental_domain}.

\begin{figure}[t]\label{fig:fundamental_domain}
    \centering
    \def\lambdaVal{4}  % Numeric value for lambda
    \pgfmathsetmacro{\Lval}{\lambdaVal/2}

\begin{tikzpicture}[scale=1.5]
  % X-axis
  \draw[->] (-3, 0) -- (3, 0) node[right] {$x$};

  % Shaded region: between the semicircle and vertical lines
  \begin{scope}
    \clip (-\Lval, 0) rectangle (\Lval, 2);          % Clip to vertical strip
    \fill[gray!30] 
      (-\Lval, 0) -- (-\Lval, 2)
      -- (\Lval, 2) -- (\Lval, 0)
      -- (1,0) arc (0:180:1) -- cycle;  % Exclude the semicircle
  \end{scope}

  % Y-axis (drawn after shading so it appears on top)
  \draw[->] (0, 0) -- (0, 2) node[above] {$y$};

  % Vertical boundary lines
  \draw[thick] (-\Lval, 0) -- (-\Lval, 2);
  \draw[thick] (\Lval, 0) -- (\Lval, 2);

  % Unit semicircle
  \draw[thick, domain=180:0] plot ({cos(\x)}, {sin(\x)});

  % Labels at numeric positions, symbolic text
  \node[below] at (-\Lval, 0) {\small $ -\frac{\lambda}{2}$};
  \node[below] at (\Lval, 0) {\small $\frac{\lambda}{2}$};
  \node[below] at (-1, 0) {$-1$};
  \node[below] at (1, 0) {$1$};
  \node at (0.9, 1) {$|z| = 1$};
\end{tikzpicture}
\caption{Fundamental domain of the Hecke group $H(\lambda)$ for $\lambda > 2$.}
\end{figure}
\par
Non holomorphic Eisenstein series and Maass cusp forms for the Hecke groups $H(\lambda)$ have been an area of extensive research over the years. See Kubota \cite{Kubota} for the definition of the non holomorphic Eisenstein series. In Knopp--Sheingorn \cite{Knopp-Sheingorn}, certain holomorphic Poincar\'e series are constructed for $H(\lambda)$ with $\lambda > 2.$ The non holomorphic Maass cusp form associated to the base eigenvalue of the Laplacian for the Hecke group $H(\lambda)$ was found by Patterson \cite{Patterson} and Sullivan \cite{Sullivan}. In the paper on Sarnak's spectral gap question \cite{Kelmer-Kontorovich-Lutsko}, Kelmer, Kontorovich, and Lutsko show that the Hecke group $H(\lambda)$ (with $\lambda >2$) has only one Maass cusp form which is square integrable on the fundamental domain. This result is originally due to Phillips and Sarnak \cite{Phillips-Sarnak}. This Maass cusp form is  the Patterson--Sullivan base eigenfunction  with non-zero Laplace eigenvalue $< \tfrac{1}{4}.$ More recently, Karlovitz \cite{karlovitz2022extensionhejhalsalgorithminfinite} generalized Hejhal's algorithm \cite{Hejhal}, for explicitly computing  Maass cusp forms on congruence subgroups of $\SL(2,\mathbb Z)$, to the case of the Hecke groups $H(\lambda)$ with $\lambda >2.$
\par
The main goal of this paper is to construct holomorphic modular forms for the Hecke groups $H\big(\sqrt{D}\big)$  where $D\equiv 1 \bmod{4}$ is a fundamental discriminant of a real quadratic field. Since the smallest possible such discriminant is $D=5$ (which means $\sqrt{D}$ is bigger than 2), it will not be possible to have holomorphic modular forms on $H\big(\sqrt{D}\big)$ which are square integrable over the fundamental domain which has infinite area. Let
 $q = e^{2\pi iz}$ with $z\in\mathfrak h$. Our approach is the construction of an $H\big(\sqrt{D}\big)$ analogue of the classical Dedekind eta function $\eta(z)$ defined by
\begin{equation} \label{DedekindEta}
\eta(z) := q^{\frac{1}{24}}\prod_{n=1}^\infty \left(1-q^n \right)
\end{equation}
which satisfies the modular relations
\begin{align}\label{Eq: Classical Modular Relations}
\eta(z+1)  = e^{\frac{\pi i}{12}}\, \eta(z), \qquad\;\; \eta(-1/z)= \sqrt{-i z}\cdot \eta(z).
\end{align}
Recall that $$\eta(z)^{24} = \Delta(z),$$ the Ramanujan cusp form of weight $12$ for $\SL(2,\mathbb Z).$ In the case of the Hecke group $H\big(\sqrt{5}\big)$, we construct an analogue of Dedekind's eta function \eqref{DedekindEta} which we denote by $\eta_5(z)$ and also an analogue of the Ramanujan cusp form $\Delta(z)$ which is denoted by $\Delta_5(z).$  Their definitions are as follows.
\begin{defn} Let $z\in\mathfrak h$ and set $q = e^{\frac{2\pi i z}{\sqrt{5}}}.$ Let $\chi_{_5}$ be the primitive real character modulo $5$. We define
\begin{align}\label{Defn: eta5}
\eta_5(z) := q^{\frac15}\cdot \prod_{n=1}^\infty \bigg ((1-q^n )^{\chi_{_5}(n)} \prod_{a=1}^4 \left(1 - e^{\frac{2\pi i a}{5}}  q^n  \right)^{\chi_{_5}(a)}\bigg).
\end{align}
Furthermore, we define
$\Delta_5(z) := \eta_5(z)^5.$
\end{defn}

The series expansion for $\Delta_5(z)$ is given by
 \begin{align}\label{Delta5 Coefficients}
 \Delta_5(z)= \sum_{N=1}^\infty \tau_5(N) q^N &= q +(-5-5\sqrt{5})q^2 + \bigg( \frac{155+45\sqrt{5}}{2} \bigg)q^3 + (-280-170\sqrt{5})q^4 \notag \\
 &\quad + (1415+490\sqrt{5})q^5+\bigg ( \frac{20565+6965\sqrt{5}}{2}\bigg) q^6 \;+\;\; \cdots.
 \end{align}

In Section \ref{sec: Proof of Modularity Theorems}, we prove the following theorem using a generalization of Weil's proof \cite{Weil} of the modularity of the Dedekind eta function.
\begin{thm} \label{Theorem: Delta5isModular} Let $z\in\mathfrak h$. The function $\Delta_5(z)$ is a holomorphic modular form for the Hecke group $H\big(\sqrt{5}\big)$, that is,
$$\Delta_5\left(\frac{az+b}{cz+d} \right) = \Delta_5(z)$$
for all $\left(\begin{smallmatrix} a&b\\c&d\end{smallmatrix}\right)\in H\big(\sqrt{5}\big)$. Furthermore,  $\Delta_5(z)$ vanishes at the cusp $\infty$, that is,
$\lim\limits_{y\to\infty} \Delta_5(iy) = 0.$
\end{thm}

\begin{rem} \label{Rem: eta5}
We prove in Section \ref{sec: Proof of Modularity Theorems} that for every
$\gamma=\left(\begin{smallmatrix} a&b\\c&d\end{smallmatrix}\right)\in H\big(\sqrt{5}\big)$, we have $$\eta_5\left(\frac{az+b}{cz+d} \right) = u_\gamma\cdot \eta_5(z)$$ where $u_\gamma$ is a certain fifth root of unity. Since $u_\gamma^5=1$, this explains why taking the fifth power of $\eta_5$ gives $\Delta_5$ which satisfies the modular relations in Theorem \ref{Theorem: Delta5isModular}. In fact, we show that $\eta_5(z)$ is invariant under the action of $S=\left(\begin{smallmatrix} 0&-1\\1&0\end{smallmatrix}\right)$. This is also illustrated in Figure \ref{fig:Modularity}.
\end{rem} 
\begin{figure}[t]
 \centering
\includegraphics[width=0.6\linewidth]{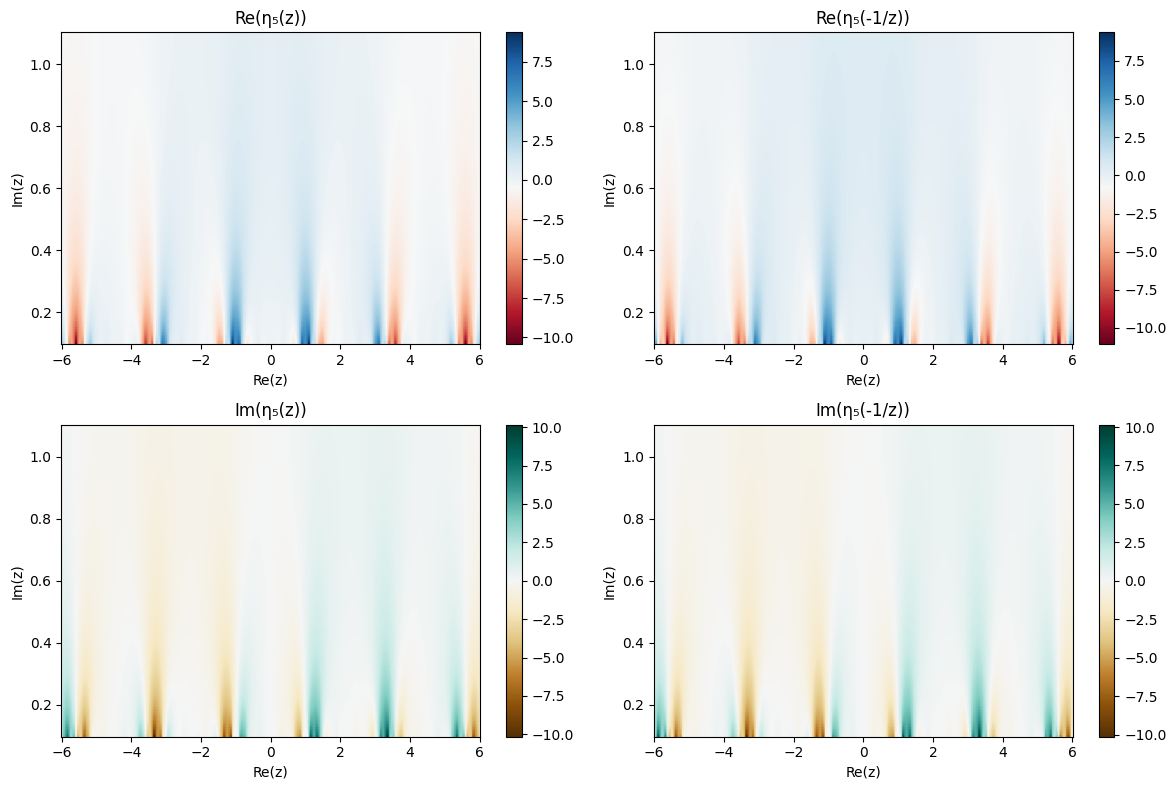}
 \caption{Modularity of $\eta_5(z)$ under the transformation $z \mapsto -1/z$, with the product over $n$ in \eqref{Defn: eta5} truncated at $n=300$. For $z = x + iy$ with $x \in (-6,6)$ and $y \in (0.1,1.1)$, the contour plots of $\operatorname{Re}(\eta_5(z))$ and $\operatorname{Re}(\eta_5(-1/z))$ coincide; likewise the contour plots of $\operatorname{Im}(\eta_5(z))$ and $\operatorname{Im}(\eta_5(-1/z))$ are identical.}
 \label{fig:Modularity}
\end{figure}
Next, we consider discriminants $D\equiv 1\bmod{4}$ with $D>5.$ In this case, we can define $\eta_{D}(z)$ as follows.

\begin{defn} Let $D\equiv 1\bmod{4}$ with $D>5$ be a fundamental discriminant of a real quadratic field. Let $\chi_D$ be the primitive real character modulo $D$. Let $\mathfrak{m}_D \in \mathbb{N}$ be such that $L(-1,\chi_D)=-2\mathfrak{m}_D$. \footnote{See Lemma \ref{lem: Value at s=-1} for a proof that $L(-1,\chi_{D})$ is a negative even number for $D>5$.} Let $z\in\mathfrak h$ and set $q = e^{\frac{2\pi i z}{\sqrt{D}}}.$ We define
\begin{align}\label{Defn: etaD}
\eta_{D}(z) := q^{\mathfrak{m}_D}\prod_{n=1}^\infty \bigg(\left(1-q^n   \right)^{{\chi_{D}}(n)}  \prod_{a=1}^D \left(1-  e^{\frac{2\pi ia}{D}} q^n  \right)^{{\chi_{D}}(a)}\bigg).
\end{align}
\end{defn}

In Section \ref{sec: Proof of Modularity Theorems}, we prove the following theorem.
\begin{thm} \label{Theorem: D>5}
Let $z\in\mathfrak h$. The function $\eta_{D}(z)$ is a holomorphic modular form for the Hecke group $H\big(\sqrt{D}\big)$, that is,
$$\eta_{D}\left(\frac{az+b}{cz+d} \right) = \eta_{D}(z)$$
for all $\left(\begin{smallmatrix} a&b\\c&d\end{smallmatrix}\right)\in H\big(\sqrt{D}\big)$. Furthermore, $\eta_{D}(z)$ vanishes at the cusp  $\infty$, that is, $\lim\limits_{y\to\infty} \eta_{D}(iy) = 0.$
\end{thm} 
\begin{comment}
\begin{rem}
\textcolor{red}{Need to edit this remark. Very crucial. Dorian: Our product has many different subproducts. It could be that one of these is a Siegel functions, but not all the time. We have constructed a holomorphic modular form for the Hecke group, he constructs it for the Siegel group.} Note that in Definitions \ref{Defn: eta5} and \ref{Defn: etaD} we set 
\[
q = q_D = e^{2 \pi i z/\sqrt{D}},
\]
which differs from the classical choice \( q = e^{2 \pi i z} \) used for the Dedekind eta function. This distinction is a key reason why the modular relations for \( \eta_D \) naturally arise in \( H(\sqrt{D}) \). In this context, we refer to the work of Brunault and Zudilin \cite[Pg. 13, Eq. 15]{brunault2023modularregulatorsmultipleeisenstein}, where eta products corresponding to the sub-products showing up in $\eta_D(z)$ appear as Siegel modular units, but for the different parameter choice $q = e^{2\pi iz}$.

\end{rem}
\end{comment}
At this stage, we make a few remarks.
\begin{rem} Curiously, we have not been able to find an analogue of the modular forms in the case when $D\equiv 3\bmod 4$. It would also be worthwhile to investigate connections between $\log  \eta_{D}(z)$, or any other function of $\eta_{D}$, with an Eisenstein series for $H(\sqrt{D})$. These constitute two interesting research problems for the future.
\end{rem}

\begin{rem}[Relation to Siegel functions]\label{rem:siegel-context}
Set $\tau=z/\sqrt D$, so that $q=e^{2\pi i z/\sqrt D}=e^{2\pi i\tau}$ is the classical $q$-parameter.
After this change of variables, the residue class subproduct over $a$ appearing in $\eta_D$
can be rewritten as products of Siegel functions (equivalently Siegel modular units)
up to explicit $q$-powers, Dedekind eta factors and roots of unity. On the other hand, the product involving the exponent $\chi_D(n)$ can be written as a ratio of products of Siegel functions by breaking $n$ into arithmetic progressions. More precisely, let $\zeta_D=e^{2\pi i/D}$ and for $0\le r,a\le D-1$, set $u=r/D$, $v=a/D$. Then the Siegel function $g_{u,v}$ has the product expansion
\[
g_{u,v}(D\tau)
= -\exp\!\Big(\pi i\,D\tau\,B_2(u)+2\pi i\,v\Big(u-\tfrac12\Big)\Big)
\prod_{m=0}^{\infty}\bigl(1-\zeta_D^{a}q^{Dm+r}\bigr)\bigl(1-\zeta_D^{-a}q^{Dm+D-r}\bigr),
\]
where $B_2(u)=u^2-u+\tfrac16$ is the second Bernoulli polynomial. We refer the reader to the standard $q$-product definition of Siegel functions (see Kubert--Lang \cite[Ch.2]{KubertLang}), which we have
specialized to the arguments $D\tau$ and $(u,v)=(r/D,a/D)$, and rewritten in terms of $q=e^{2\pi i\tau}$. See also \cite[Eq.~(15), Pg.~13]{brunault2023modularregulatorsmultipleeisenstein}.

We emphasize, however, that our modularity statements in Theorems \ref{Theorem: Delta5isModular} and \ref{Theorem: D>5} are formulated for the Hecke group
$H(\sqrt D)=\langle z\mapsto z+\sqrt D,\ z\mapsto -1/z\rangle$ in the $z$-variable.
In $\tau$-coordinates, the inversion corresponds to the map $\tau\mapsto -1/(D\tau)$,
which lies outside $\mathrm{SL}_2(\mathbb Z)$ and thus the standard transformation law $g_{(u,v)}(\gamma\tau) \to g_{(u,v)\gamma}(\tau)$ for $\gamma\in\mathrm{SL}_2(\mathbb Z)$ cannot be directly applied. Moreover, we show in Lemma \ref{Thm : Modular Relation}, the product over $a$ acts a dual function to the product with exponent $\chi_D(n)$ under the inversion $z\mapsto -1/z$ which plays a key role in making $\eta_D(z)$ modular. One could possibly attempt to establish modularity of $\eta_D(z)$ by studying each Siegel modular unit separately. But such an approach would likely require tracking explicitly the action of inversion on each factor arising from $(u,v)$. Our Weil style argument instead gives a direct and compact verification of the transformation laws for the Hecke group $H(\sqrt{D})$.
\end{rem}

\begin{rem} In the Hecke group $H(\sqrt{D})$ with $D\equiv 1\bmod{4}$ and $D\geq 5$, there is a hyperbolic element $\gamma_D = \left(\begin{smallmatrix} \sqrt{D}&-1\\1&0\end{smallmatrix}\right)$ with the fixed points
$$
\omega_D=\frac{\sqrt{D}+\sqrt{D-4}}{2}, \quad \textrm{and} \quad \omega_D^{\prime}=\frac{\sqrt{D}-\sqrt{D-4}}{2},
$$
that is, $\gamma_D \omega_D=\omega_D, \gamma_D \omega_D^{\prime}=\omega^{\prime}_D$. Consider the matrix $W_D=\left(\begin{smallmatrix} 1&-\omega_D\\1&-\omega_D^{\prime}\end{smallmatrix}\right)
\in \operatorname{SL}(2, \mathbb{R})$ which has the property that the conjugate of $\gamma_D$ by $W_D$ is a diagonal matrix. It follows that
$$
W_D \cdot \gamma_D \cdot W_D^{-1}=\left(\begin{array}{cc}
\sqrt{\kappa} & 0 \\
0 & 1 / \sqrt{\kappa}
\end{array}\right),
$$
with $\sqrt{\kappa}+1/\sqrt{\kappa}=\sqrt{D}$ since conjugation preserves the trace of a matrix. Then if $f(z)$ is modular for the group $H(\sqrt{D})$, the function $f\left(W_D^{-1} z\right)$ is invariant under the transformation $z \mapsto \kappa z$. Following \cite{Kontorovich}, a flare domain for the conjugated Hecke group $H(\sqrt{D})$ is isometric to a domain of the form $\{z \in \mathfrak{h}: 1<|z|<\kappa, \, 0<\arg (z)<\alpha\}$ for some angle $\alpha<\tfrac{\pi}{2}$. In this paper, the functions $\Delta_5(z)$ and $\eta_D(z)$ (with $D\equiv 1\bmod{4}$ and $D> 5$) are holomorphic functions which are modular for $H(\sqrt{D})$, vanish at the cusp at $\infty$, but are not square integrable in the flare domain.
\end{rem}
Theorems \ref{Theorem: Delta5isModular} and \ref{Theorem: D>5}, alongside Remark \ref{Rem: eta5} show that for $z \in \mathfrak{h}$, the function $\eta_{D}(z)$ admits the following Fourier expansion:
\begin{align}\label{Fourier Coefficients Definition}
\eta_{D}(z) = q^{-\frac{L(-1,\chi_{D})}{2}} \sum_{N=0}^{\infty} a_{D}(N) \,  q^{N}.
\end{align}
Here we have noted that $L(-1,\chi_5) = -\frac{2}{5}$. We refer the reader to Section \ref{sec: Graphical Data} where we present a table containing the exact values of \(a_{D}(N)\) in the cases \(D = 5, 13, 17\). A natural question at this stage is to understand the size of the Fourier coefficients $a_{D}(N)$. The growth of these coefficients is highly sensitive to the structure of the infinite product defining the modular form, and changes in the exponents can dramatically alter their behavior.

To appreciate the range of possibilities, it is instructive to compare two classical extremes. For the Dedekind eta function $\eta(z)$, the exponents of the factors in the product over $n$ in \eqref{DedekindEta} are all equal to $1$. In this setting, the Fourier coefficients $\tau(N)$ that arise from $\Delta(z) = \eta(z)^{24}$ are multiplicative and satisfy Deligne's bound \cite{Deligne74}, that is, 
\[
 \lvert \tau(p) \rvert  \leq 2 p^{\frac{11}{2}}, \quad \textrm{for all primes } p.
\] 
In particular, the coefficients show polynomial growth.

At the opposite extreme, if the exponents are all equal to $-1$ as in the case of $\eta(z)^{-1}$, the resulting coefficients are given by $p(N)$, the number of partitions of $N$. In this context, Hardy and Ramanujan \cite{Hardy-Ramanujan} showed that $p(N)$ exhibits exponential growth, that is,
\[
p(N) \sim \frac{1}{4 N \sqrt{3}} \exp \bigg ( \pi \sqrt{\frac{2}{3}} N^{\frac{1}{2}} \bigg), \quad \textrm{as} \quad N \to \infty.
\]
\par
The functions $\eta_D(z)$ defined in \eqref{Defn: eta5} and \eqref{Defn: etaD} occupy an intermediate position between these two extremes. Their defining infinite products involve exponents taking the values $1$ and $-1$, according to the real character $\chi_D$. This raises the question of whether the coefficients $a_D(N)$ exhibit polynomial growth, or whether the presence of negative exponents in the product formula instead leads to exponential growth. As a numerical illustration, we refer the reader to Figure \ref{fig:D=5 Pic1}, which appears to suggest that in the case $D=5$, $\log |a_5(N)|$ grows linearly with $\sqrt{N}$.
\par
\begin{figure}[t]
    \centering
    \includegraphics[width = .45\textwidth]{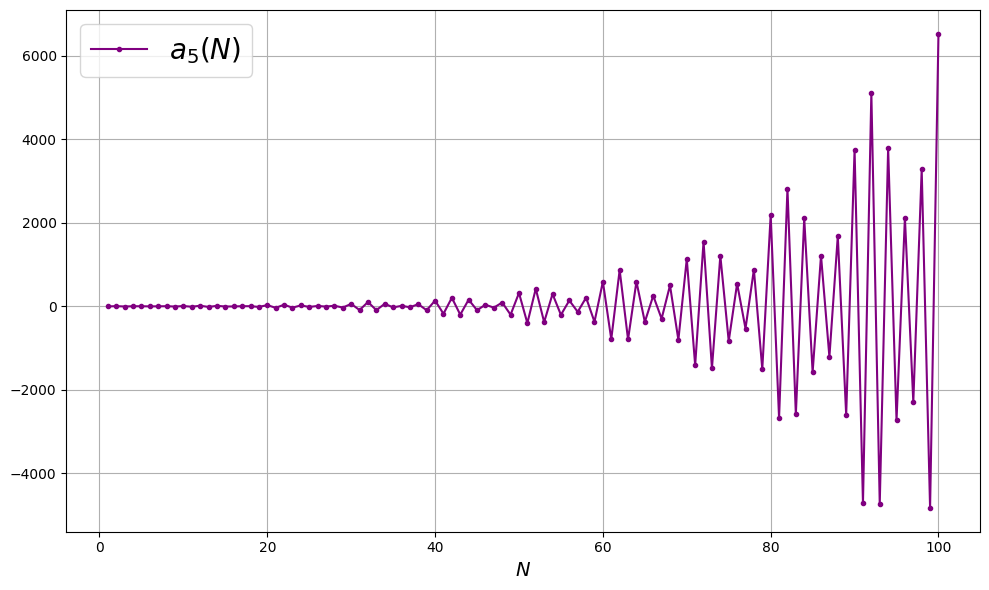}
 \includegraphics[width=0.45\linewidth]{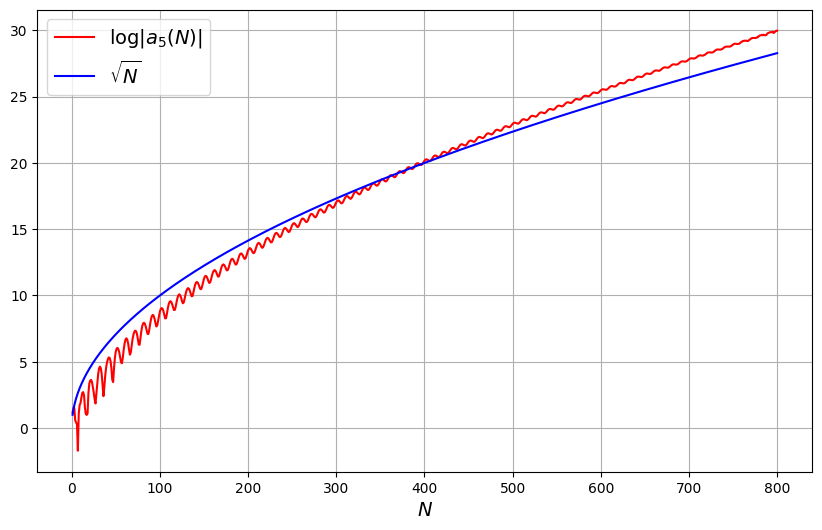}
\caption{Left: Plot of $a_5(N)$, $1 \leq N \leq 100$. Right: Growth of $\log \, \lvert a_5(N) \rvert $ (in red) compared to $\sqrt{N}$ (in blue), for $N =1$ to $800$.}
    \label{fig:D=5 Pic1}
\end{figure}
A second and equally fundamental question concerns the sign behavior of $a_D(N)$. Even when the growth rate of coefficients is understood, their oscillatory nature often encodes deeper arithmetic information. Here the functions $\eta_D(z)$ are built from infinite products whose exponents take both positive and negative values, and the factors also incorporate certain roots of unity. It is therefore not a priori clear whether the resulting Fourier coefficients $a_D(N)$ should display persistent sign changes or exhibit a dominant sign bias. A precise formulation of this question is to determine whether $a_D(N)$ changes sign infinitely often, and more generally, to quantify the frequency and distribution of such sign changes. We refer the reader to Figures \ref{fig:D=5 Pic1} and \ref{fig:Fourier Sequence}, where the plots for the sequence $a_D(N)$ for $D=5,13$ and $17$ display remarkably distinct patterns of oscillations and sign changes.
\par
In what follows, our objective is to precisely understand the arithmetic behavior of the Fourier coefficients $a_D(N)$ and answer both the questions raised above. To state our results, we need some notation. Let $\mf{n}_D$ be the least quadratic non-residue modulo $D$. Note that $\mf{n}_D$ is prime. The singular series, which will determine the sign patterns of $a_D(N)$, is given as follows. Define
\begin{equation}\label{sin series}
\mf{S}_D(N)
:=
\begin{cases}
(-1)^N,&\qquad\qquad \n=2,
\\
2\displaystyle\sum_{u=1}^{(\n-1)/2}\cos\bigg(\frac{2\pi uN}{\n}-\beta_D(u,\mf{n}_D)\bigg),
&\qquad\qquad \n\geq 3,
\end{cases}
\end{equation}
where
\begin{align}\label{Beta def}
\beta_D(u,\mf{n}_D)& : =
\frac{2\pi}{\phi(\mf{n}_D)} \operatorname{Re}\sideset{}{'}\sum_{\substack{\psi \bmod \mf{n}_D \\ \psi \textnormal{ complex, odd}}} \psi(u) \Big[
L(0,\overline{\psi})L(0,\chi_D\psi)-\psi(D) L(0,\psi)L(0,\chi_D\overline{\psi})\Big] \notag \\
&\quad + \mathbbm{1}_{\n \equiv 3\bmod 4} \cdot \frac{2\pi}{\phi(\mf{n}_D)}\chi_{\n}(u) L(0,\chi_{\n})L(0,\chi_D\chi_{\n}).
\end{align}
Here the primed sum indicates that we only take one representative from each pair of complex conjugate odd characters and $\chi_{\n}$ is the real primitive character modulo $\n$. By the functional equation, $\beta_D(u,\mf{n}_D)$ can be expressed in terms of the Dirichlet $L$-values at $s=1$ which is more computationally practical via the class number formula (this is given in Lemma \ref{L1 lem}).
\begin{figure}[t]
    \centering
    \includegraphics[width = .45\textwidth]{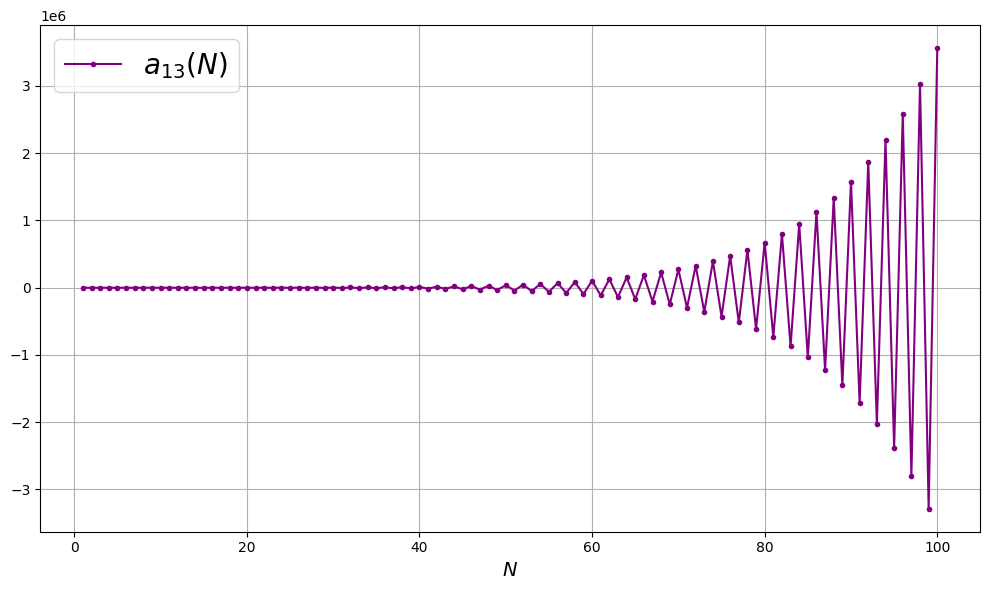}
 \includegraphics[width = .45\textwidth]{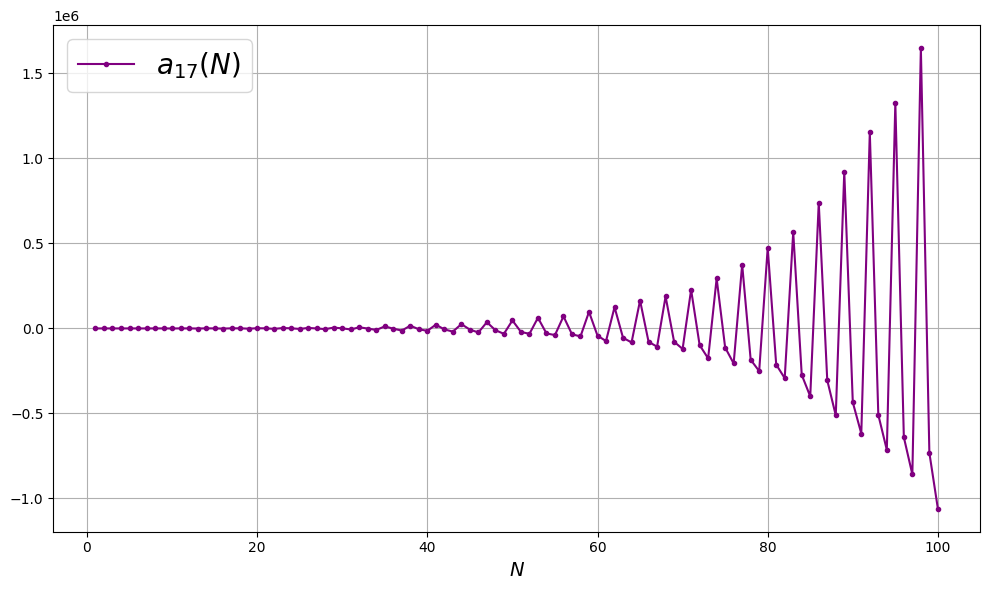}

    \caption{Left: Plot of $a_{13}(N)$, $1 \leq N \leq 100$. Right: Plot of $a_{17}(N), 1 \leq N \leq 100$. The sign pattern for $a_{13}(N)$ shows alternate positive and negatives. In contrast, for $a_{17}(N)$ the pattern follows two negatives followed by a positive sign.}
    \label{fig:Fourier Sequence}
\end{figure}
We are now ready to state our results.
\begin{thm}\label{gen thm}
Let $D\equiv 1\bmod{4}$ with $D>0$ be a fundamental discriminant of a real quadratic field. Let $\chi_D$ be the primitive real character modulo $D$, and $a_D(N)$ be given by \eqref{Fourier Coefficients Definition}. Then $a_{D}(N)$ lies in the ring of integers $\mathcal{O}_K= \mathbb{Z}\left[\frac{1 + \sqrt{D}}{2}\right] $ of the quadratic field $K=\mathbb{Q}(\sqrt{D})$. Furthermore, define 
\[
c_D
=\frac{\sqrt{D}L(2,\chi_D)}{\mf{n}_D^2},
\]
and
\[
\alpha_D=\sqrt{D}L(1,\chi_D)+\frac14\log c_D-\frac12\log(4\pi).
\]
Then if $\n<\sqrt{D}$, we have for $N \geq 1$,
\begin{align}\label{Eq: Key Asymptotic}
a_D(N)=\frac{e^{\alpha_D}\mf{S}_D(N)}{N^{3/4}}\exp  ({2\sqrt{c_DN}})(1+\mathcal{O}_D(N^{-\frac14})),
\end{align}
where $\mf{S}_D(N)$ is given by \eqref{sin series} and the implied constant in the error term depends at most on $D$.
\end{thm}

For clarity, we separately state the result for the special cases $\mf{n}_D=2,3$ which is an immediate corollary of Theorem \ref{gen thm}. Note that when $\n=3$, there is only one non-principal character modulo $3$ and hence the sum in $\beta_D(u,\mf{n}_D)$ is empty. 
 \begin{coro}\label{main thm}
Let the notations be as in Theorem \ref{gen thm}. If $\mf{n}_D=2$, we have for $N\geq 1$,
\[
a_D(N)=\frac{(-1)^Ne^{\alpha_D}}{N^{3/4}}\exp\big({2\sqrt{c_DN}}\big)(1+\mathcal{O}_D(N^{-\frac14})).
\]
If $\mf{n}_D=3$, we have for $N \geq 1$, 
\[
a_D(N)=\frac{2\cos(\frac{2\pi}{3}N-\beta)e^{\alpha_D}}{N^{3/4}}\exp\big({2\sqrt{c_DN}}\big)(1+\mathcal{O}_D(N^{-\frac14})),
\]
where 
\begin{align}\label{Eq: Beta for n_D=3}
\beta
=
\pi L(0,\chi_3)L(0,\chi_D\chi_3)
=
\frac{3\sqrt{D}}{\pi}L(1,\chi_3)L(1,\chi_D\chi_3),
\end{align}
and $\chi_3$ is the primitive real character modulo 3. 
\end{coro}
We make the following remarks.
\begin{rem}\label{Rem: Least Quadratic Nonresidue}
We do not pursue the case $\n \ge \sqrt{D}$ in Theorem \ref{gen thm}, since we believe it to be never true in our context. In fact, when $D$ is prime, we expect substantially stronger bounds of the form $\mf{n}_p \ll_{\varepsilon} p^{\varepsilon}$ following Vinogradov's conjecture. Assuming the Generalized Riemann Hypothesis, it is known that $\mf{n}_p \le (\log p)^2$, see \cite{LLS}. Unconditionally, Leveque (see \cite{Lev}, pp.~122--123) proves the inequality
$\mf{n}_p < \sqrt{p}$
for all primes $p \neq 2,3,7,23$. This settles the case when $D$ is prime in Theorem \ref{gen thm}.

However, it seems difficult to establish the inequality $\mf{n}_D < \sqrt{D}$ for all composite moduli. By \cite{Bordignon} and \cite{Khale}, such results are known for all $D \geq \exp(\exp(
{9.594})$ say. However the lower bound on $D$ for which these results are valid are
somewhat astronomical and beyond computational limits, and therefore, it is not possible to verify the remaining values numerically.
\end{rem}
\begin{figure}[t]
    \centering
\includegraphics[width = .45\textwidth]{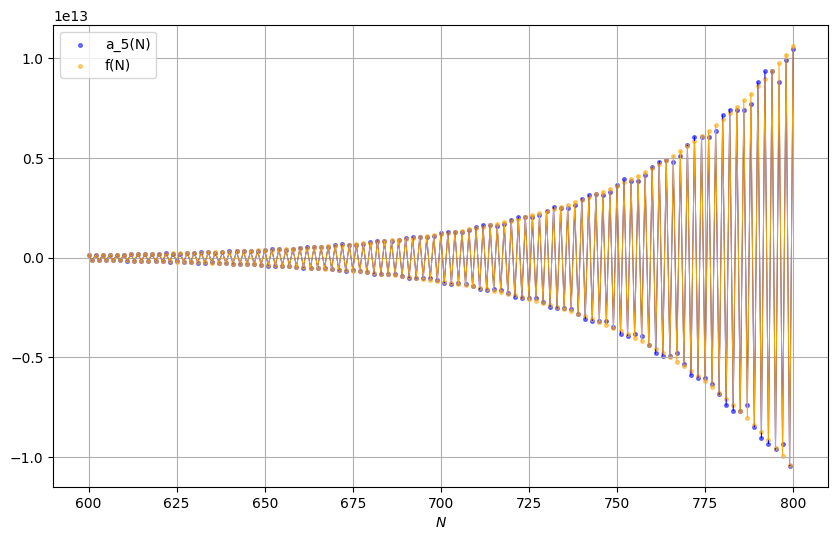}
 \includegraphics[width = .45\textwidth]{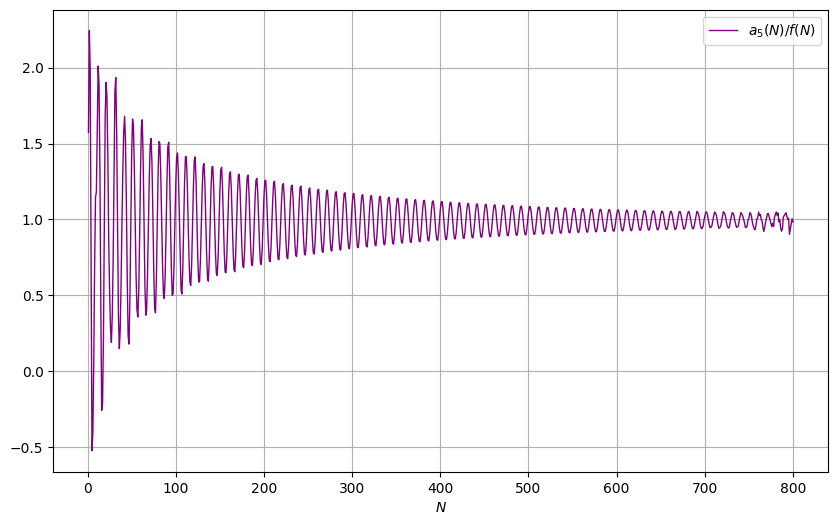}

    \caption{Left: Plot of $a_{5}(N)$ (in blue) against $f(N)$ (in orange), $600 \leq N \leq 800$, where $f(N) = (-1)^N \phi^2( N^{3/4}2\sqrt{5})^{-1} \exp(2 \pi \sqrt{N}/5)$ and $\phi = (1+\sqrt{5})/2$. The plotted values show close agreement. Right: Plot of $a_{5}(N)/f(N)$, $1 \leq N \leq 800$. The ratio tends to $1$ as $N$ becomes large.}
    \label{fig:a5comparison}
\end{figure}
\begin{rem}
When $D=5$, we have $\mf{n}_D=2$ and the special values 
\[
L(1,\chi_D)=\frac{2}{\sqrt{5}}\log \phi,\qquad L(2,\chi_D)=\frac{4\pi^2}{25\sqrt{5}},
\]
where $\phi=(1+\sqrt{5})/2$. In this case, $c_D=\pi^2/25$ and we have the asymptotic relation
\begin{align}\label{Eq: a5 Asymptotic}
a_5(N)=\frac{(-1)^N}{N^{3/4}} \cdot \frac{\phi^2}{2\sqrt{5}}\exp\bigg(\frac{2\pi\sqrt{N}}{5}\bigg) (1+\mathcal{O}(N^{-\frac14})).
\end{align}
We refer the reader to Figure \ref{fig:a5comparison} which compares the two sides of \eqref{Eq: a5 Asymptotic}.
\end{rem}
The sign changes of the Fourier coefficients $a_D(N)$ are governed by the singular series $\mf{S}_D(N)$. We can see from the formula for $\mf{S}_D(N)$ in \eqref{sin series}, that the sign patterns are periodic modulo $\n$. In specific cases, we can precisely compute the values of $\beta_D(u,\mf{n}_D)$ and therefore, the singular series $\mf{S}_D(N)$ using either
\[
L(0,\chi_D)=-\frac{1}{D}\sum_{a=1}^{D-1} \chi_D(a)a,
\] 
or by using the class number formula. This shall give a complete picture of the sign patterns. For instance, recall from Corollary \ref{main thm} that in the case $\n=3$, we have
\[
\mf{S}_{D}(N) = 2 \cos\bigg(\frac{2\pi}{3}N-\beta\bigg), \quad \textrm{where} \quad \beta = \pi L(0,\chi_3)L(0,\chi_D\chi_3).
\]
By evaluating $\beta$, we find that the sign patterns in this case takes a very simple form of a three periodic sequence which explains the nature of the graph for $a_{17}(N)$ in Figure \ref{fig:Fourier Sequence} (note that $\mf{n}_{17} =3$). We also compute an example case for $\n=5$, the first modulus for which this occurs being $D=73$. We present these results in the following proposition.
\begin{prop}\label{n3 prop}
Let the notations be as in Theorem \ref{gen thm}. For $\mf{n}_D=3$, the expression $\cos(\tfrac{2\pi}{3}N-\beta)$ takes the values 
\[
1,-\frac12,-\frac12
\]
periodically, where $\beta$ is given by \eqref{Eq: Beta for n_D=3}. In particular, for $D=17, 41$, we have
\[
\Sigma_D(N) = 2 \cos\bigg(\frac{2\pi}{3}N-\beta\bigg)=2,-1,-1
\]
for $N\equiv 1,2,3\bmod 3$ respectively. For $D=73$, we have for $N \geq 1$, 
\[
\mf{S}_{73}(N)
=
2\cos\bigg(\frac{2\pi}{5}N-\frac{6\pi}{5}\bigg)+2\cos\bigg(\frac{4\pi}{5}N+\frac{6\pi}{5}\bigg)
\]
which takes the values
\[
\frac{3-\sqrt{5}}{2}, -1, \frac{3+\sqrt{5}}{2}, -1+\sqrt{5}, -1-\sqrt{5}
\]
for $N\equiv 1,2,3,4,5 \bmod 5$ respectively. 
\end{prop}  

We end the introduction with a more detailed analysis of the case $D=5$. Note that the oscillations of $a_5(N)$ in Figure \ref{fig:D=5 Pic1} take a much more jagged appearance compared to those of $a_{13}(N)$ and $a_{17}(N)$ in Figure \ref{fig:Fourier Sequence}. As we shall see, this is due to the influence of secondary exponential terms. When $D=5$, these secondary terms are much closer in size to the leading term and therefore, their effect is more pronounced, particularly when $N$ is relatively small. They also include an oscillatory sign factor which is out of phase with that of the leading sign factor $(-1)^N$, thereby accounting for the jagged appearance. The following result states an asymptotic formula for $a_D(5)$ which includes these influential secondary exponential terms. 

\begin{thm}[Two term asymptotic expansion]\label{mod 5 thm}
Let the notations be as in Theorem \ref{gen thm}. Let $N \geq 1$ and $\phi = (1+\sqrt{5})/2$. For each fixed $J\ge 1$, we have 
\begin{align}\label{eq:two-cusp-theorem}
a_5(N)& =\frac{(-1)^N}{N^{3/4}} \frac{\phi^2}{2\sqrt5}\,
\exp\!\Big(\frac{2\pi}{5}\sqrt N\Big)\,
\bigg(1+\sum_{j=1}^J b_j\,N^{-j/2}+\mathcal{O}_J\big(N^{-(J+1)/2}\big)\bigg)\nonumber\\
&\quad+\frac{1}{N^{3/4}}\frac{\sqrt2\,\phi^2}{5^{3/4}}\,
\cos \Big(\frac{4\pi}{5}N+\frac{8\pi}{25}\Big) \exp \Big(\frac{4\pi}{5\sqrt5}\sqrt N\Big)\,
\,
\bigg(1+\sum_{j=1}^J \widetilde{b_j} \, N^{-j/2}+\mathcal{O}_J\big(N^{-(J+1)/2}\big)\bigg)\nonumber\\
&\quad+ \mathcal{O}\left(N^{-\frac34}\exp\!\Big(\frac{4\pi}{15}\sqrt N\Big)\right),
\end{align}
   where $b_j, \widetilde{b_j}$ are effectively computable absolute constants for $1 \leq j \leq J$. 
\end{thm}

\begin{rem}
When $D=5$, the leading exponential terms in \eqref{eq:two-cusp-theorem} are
\[
\exp\bigg( \frac{2 \pi \sqrt{N}}{5} \bigg) \quad \textrm{and} \quad  \exp\bigg( \frac{4 \pi \sqrt{N}}{5\sqrt{5}} \bigg),  \quad \textrm{where} \quad \frac{2 \pi}{5} = 1.2566\dots, \quad \frac{4 \pi}{5\sqrt{5}} =1.1239\dots.
\]
As we shall show later (see Remark \ref{Comparion Remark}), the first two leading exponential terms in the asymptotic expansion of $a_{13}(N)$ are given by
\[
\exp\bigg( \frac{2 \pi \sqrt{N}}{\sqrt{13}} \bigg) \quad \textrm{and} \quad  \exp\bigg( \frac{4 \pi \sqrt{N}}{13} \bigg),  \quad \textrm{where} \quad \frac{2 \pi}{\sqrt{13}} = 1.7426\dots, \quad \frac{4 \pi}{13} =0.9666\dots
\]
Therefore, in the case $D=5$, the secondary exponential term, alongside the cosine factor in the front, has a significantly more pronounced influence on the behavior of $a_5(N)$ than in the case $D=13$, particularly when $N$ is relatively small. This accounts for the difference in the graphs of $a_5(N)$ and $a_{13}(N)$ shown in Figures \ref{fig:D=5 Pic1} and \ref{fig:Fourier Sequence}.
\end{rem}

One can, in principle, acquire an infinite asymptotic series of decreasing exponential terms in \eqref{eq:two-cusp-theorem}. For $a_D(N)$ in general, the lower order terms become rather complex, which is why we have restricted our attention to just two terms in the $D=5$ case.

Finally, we highlight an exact formula connecting the Fourier coefficients $a_{D}(N)$ with classical partition functions, which we establish in Section \ref{sec: Fourier Coefficients} and which may be of independent interest. In particular, the twisted partition function
\begin{align}\label{Twisted Partition function}
p_{\textrm{ord}}(k,\theta) = \sum_{\lambda \vdash k} \theta^{\ell(\lambda)}, \quad \theta \in \mathbb{C}, \, |\theta| \le 1,
\end{align}
plays a central role in this correspondence. Here $\lambda$ denotes a partition of $k$ and $\ell(\lambda)$ its length, that is, the number of parts in the partition. When $\theta = 1$, this reduces to the ordinary partition function. In our setting, $\theta$ will be taken to be certain $D$-th roots of unity; we refer the reader to Section~\ref{sec: Fourier Coefficients} for a detailed account of this connection. The study of recurrence relations and congruence properties of the sequence $a_D(N)$ similar to the case of partitions is also of interest, and may be a topic of discussion on another occasion.

\subsection*{Structure of the paper} In Section \ref{sec: Phi Modular Relation}, we construct  certain eta-products twisted by real characters and derive modular relations for them. Using these modular relations, we prove Theorems \ref{Theorem: Delta5isModular} and \ref{Theorem: D>5} in Section \ref{sec: Proof of Modularity Theorems}. Sections \ref{Sec: Main Theorem Proof I} and \ref{Sec: Main Theorem Proof II} are devoted to the study of asymptotic behavior and sign changes of $a_D(N)$, leading to the proofs of Theorem \ref{gen thm} and Proposition \ref{n3 prop}. In Section \ref{Sec: Sketch Proof}, we provide a sketch of the proof of Theorem \ref{mod 5 thm}, building on the ideas from Sections \ref{Sec: Main Theorem Proof I} and \ref{Sec: Main Theorem Proof II}. In Section \ref{sec: Fourier Coefficients}, we explore some interesting connections between the Fourier coefficients \(a_{D}(N)\) and the theory of partitions. Finally, in Section \ref{sec: Graphical Data}, we present a table containing the exact values of \(a_{D}(N)\) in the cases \(D = 5, 13, 17\) for $1 \leq N \leq 25$.
\subsection*{General Notations.}
We employ some standard notation that will be used throughout the paper.

\begin{itemize}
    \item Throughout the paper, the expressions $f(X)=\mathcal{O}(g(X))$, $f(X) \ll g(X)$, and $g(X) \gg f(X)$ are equivalent to the statement that $|f(X)| \leq C|g(X)|$ for all sufficiently large $X$, where $C>0$ is an absolute constant. A subscript of the form $\ll_{\alpha}$ means the implied constant may depend on the parameter $\alpha$. Dependence on several parameters is indicated in an analogous manner, as in $\ll_{\alpha, \lambda}$.
    \item Throughout the paper, we shall let $D\equiv 1\bmod{4}$ with $D>0$ be a fundamental discriminant of a real quadratic field and $\chi_D$ shall denote the primitive real character modulo $D$. For $D>5$, we set $\mathfrak{m}_D \in \mathbb{N}$ be such that $L(-1,\chi_D)=2\mathfrak{m}_D$.
    \item We denote the least quadratic non-residue modulo $D$ by $\mathfrak{n}_D$ and $\chi_{\mathfrak{n}_D}$ shall be the primitive real character modulo $\mathfrak{n}_D$.
    \item For a primitive character $\chi \bmod q$, the Gauss sum is denoted by $\tau(\chi)$.
    \item The function $\phi$ denotes the Euler totient function. We also use the same notation to denote the golden ratio $(1+\sqrt{5})/2$, and this should be clear from the context.
    \item We write $\zeta_D$ to denote the $D$-th root of unity $\zeta_D = e^{\frac{2 \pi i }{D}}$.
    \item We set $e(x) = e^{2\pi i x}$ and use both notations interchangeably throughout the paper.
    \item We use the notation $n \equiv \square \bmod D$ to indicate that $n$ is a square modulo $D$. Analogously, the notation $n \not \equiv \square \bmod D$ means that $n$ is not a square modulo $D$.
    \item We denote by $\varepsilon$ an arbitrarily small positive quantity that may vary from one line to the next, or even within the same line. Thus we may write $X^{2 \varepsilon} \leq X^\varepsilon$ with no reservations.
    
\end{itemize}

\section{Modular Relations between Twisted Eta Products} \label{sec: Phi Modular Relation} Let $D\equiv 1\bmod{4}$ with $D>0$ be a fundamental discriminant of a real quadratic field. Let ${\chi_{D}}$ be the primitive real character modulo $D$. For $z \in \mathfrak{h}$, define the twisted product
\begin{align}\label{Eq: Phi Chi}
\Phi_{D}(z) := \prod_{n=1}^\infty \left(1-q^n   \right)^{{\chi_{D}}(n)}, \quad \textrm{where} \quad q = e^{\frac{2 \pi i z}{\sqrt{D}}}.
\end{align}
The Dirichlet $L$-function 
$L(s,{\chi_{D}}) =\sum\limits_{n=1}^\infty {\chi_{D}}(n) n^{-s}$ satisfies the functional equation (see \cite[Ch. 5]{IK2004})
\begin{align} \label{FunctionalEquationForL(s,chi)}
\Lambda(s,{\chi_{D}}) := \left(\frac{D}{\pi} \right)^{\tfrac{s}{2}} \Gamma\left(\frac{s}{2} \right) L(s,{\chi_{D}}) =  \Lambda(1-s,{\chi_{D}}).
\end{align}
In this section, we prove the following theorem which gives a modular relation between $\Phi_{D}$ and a certain dual of $\Phi_{D}$, denoted by $\Phi^\#_{D}$, which is defined as follows. For $z \in \mathfrak{h}$, let
\begin{align}\label{Eq: Dual Phi}
\Phi^\#_{D}(z) := \prod_{n=1}^\infty \prod_{a=1}^D \left(1-  e^{\frac{2\pi ia}{D}}\cdot q^n  \right)^{{\chi_{D}}(a)}.
\end{align}
\begin{thm} \label{Thm : Modular Relation} For $z\in \mathfrak{h}$, we have
\begin{align}\label{Thm: Modular Relation}
\Phi^\#_{D}(-1/z) = e^{L'(0,{\chi_{D}})-iz\cdot\frac{\pi \, L(-1,{\chi_D})}{\sqrt{D}}}\cdot \Phi_{D}(z),
\end{align}
where $\Phi_{D}(z)$ and $\Phi^\#_{D}(z)$ are defined by \eqref{Eq: Phi Chi} and \eqref{Eq: Dual Phi} respectively.
\end{thm}
We break the proof of Theorem \ref{Thm : Modular Relation} into the following subsections.
\subsection{The function $F_D(z)$.} 
By Taylor series expansion, we have
\[
\log(1-x) = -\sum\limits_{m=1}^\infty m^{-1} x^m, \quad \quad |x|<1.
\]
For $z \in \mathfrak{h}$, define
 \begin{align}
 F_D(z) & := \sum_{m=1}^\infty \sum_{n=1}^\infty \frac{{\chi_{D}}(n)}{m} e^{\frac{2\pi i mnz}{\sqrt{D}}}  =\sum_{n=1}^\infty {\chi_{D}}(n)\bigg (\sum_{m=1}^\infty m^{-1} \left(e^{\frac{2\pi i mnz}{\sqrt{D}}}\right)^m    \bigg )\notag \\
 & = -\sum_{n=1}^\infty {\chi_{D}}(n)\log\left (1 - e^{\frac{2\pi i mnz}{\sqrt{D}}} \right) =  - \log \Phi_{D}(z). \label{Fchi}
 \end{align}
Let $s\in\mathbb C$ with ${\rm Re}(s)>1$. Then for $\sigma>1,$
we have the Mellin transform pair
\begin{align*}
\widetilde{F}_D(s) &:= \int_0^\infty F_{D}(iu)\, u^s \, \frac{du}{u}, \quad \textrm{and} \quad
F_{D}(iu)= \frac{1}{2\pi i} \int_{\sigma-i\infty}^{\sigma+i\infty} \widetilde{F}_{D}(s) u^{-s} \, ds, \notag
\end{align*}
where $u>0$. We prove the following lemma.    
\begin{lem} \label{FunctionalEquationTildeFchi} Let $s\in\mathbb C$ with ${\rm Re}(s)>1$. The function  $\widetilde {F}_{D}(s)$ satisfies the functional equation
$$\widetilde {F}_{D}(s) = \bigg(\frac{\sqrt{D}}{2\pi}\bigg)^{s} \Gamma(s) L(s,{\chi_{D}}) \zeta(s+1) = \sqrt{D}\bigg(\frac{\sqrt{D}}{2\pi}\bigg)^{-s}\Gamma(-s)\,  \zeta(-s)\, L(1-s,{\chi_{D}}).$$
  \end{lem} 
  \begin{proof} 
  We compute
  \begin{align*}
 \widetilde {F}_{D}(s) & = \int_0^\infty  \sum_{m=1}^\infty \sum_{n=1}^\infty \frac{{\chi_{D}}(n)}{m} e^{-\frac{2\pi  mnu}{\sqrt{D}}} \, u^s \,\frac{du}{u} =  \sum_{m=1}^\infty \sum_{n=1}^\infty  \frac{{\chi_{D}}(n)}{m} \int_0^\infty   e^{-\frac{2\pi  mnu}{\sqrt{D}}} \, u^s \,\frac{du}{u}\\
  & =  \sum_{m=1}^\infty \sum_{n=1}^\infty  \frac{{\chi_{D}}(n)}{m}\cdot \sqrt{D}^s\, (2\pi m n)^{-s} \Gamma(s) = \bigg(\frac{\sqrt{D}}{2\pi}\bigg)^{s} \Gamma(s) L(s,{\chi_{D}}) \zeta(s+1).
 \end{align*}   
We have the functional equation of the Riemann zeta function
$$\Lambda(s) := \pi^{-\frac{s}{2}}  \Gamma\left(\tfrac{s}{2} \right) \zeta(s) =\Lambda(1-s).$$
It follows from \eqref{FunctionalEquationForL(s,chi)} and the identity $\Gamma\left(\frac{s}{2}\right) \Gamma\left(\frac{s+1}{2}\right)= 2^{1-s} \sqrt{\pi}\,\Gamma(s)$ that
$$\Lambda(s+1)\Lambda(s,{\chi_{D}}) = 2\, D^{\frac{s}{2}} (2\pi)^{-s} \Gamma(s) \zeta(s+1) L(s,{\chi_{D}}) = \Lambda(-s)\Lambda(1-s,{\chi_{D}}).$$
Then we have the functional equation
  $$\widetilde {F}_{D}(s) = \sqrt{D}\bigg(\frac{\sqrt{D}}{2\pi}\bigg)^{-s}\Gamma(-s)\,  \zeta(-s)\, L(1-s,{\chi_{D}})
 $$
 which completes the proof.
\end{proof}
\subsection{The dual function $F_{D}^\#(z)$}  
We recall the identity (see \cite[Ch. 3]{IK2004})
\[
{\chi_{D}}(m) = \frac{1}{\sqrt{D}} \, \sum\limits_{a=1}^D {\chi_{D}}(a) e^{\frac{2\pi i am}{D}},
\]
which gives the Fourier expansion of ${\chi_{D}}$ in terms of additive characters. For $z \in \mathfrak{h}$, define the dual function
  \begin{align}
  F_{D}^\#(z) & := \sqrt{D}\sum_{m=1}^\infty\sum_{n=1}^\infty \frac{{\chi_{D}}(m)}{m} e^{\frac{2\pi i m n z}{\sqrt{D}}} = \sum_{a=1}^D {\chi_{D}}(a)\sum_{n=1}^\infty \bigg( \sum_{m=1}^{\infty} \frac{1}{m}\left(e^{\frac{2\pi ia}{D}}\cdot e^{\frac{2\pi i  n z}{\sqrt{D}}}  \right)^m \bigg) \notag \\
  & = -  \sum_{a=1}^D {\chi_{D}}(a)\sum_{n=1}^\infty \log\left(1-  e^{\frac{2\pi ia}{D}}\cdot e^{\frac{2\pi i  n z}{\sqrt{D}}} \right)=  - \log \Phi_{D}^\#(z). \label{FchiSharp}
  \end{align}
Let $s\in\mathbb C$ with ${\rm Re}(s)>1$. Then for $\sigma>1,$
we have the Mellin transform pair
\begin{align*}
\widetilde {F}^\#_{D}(s) &:= \int_0^\infty F^\#_{D}(iu)\, u^s\,\frac{du}{u}, \quad \textrm{and} \quad F^\#_{D}(iu)= \frac{1}{2\pi i} \int_{\sigma-i\infty}^{\sigma+i\infty} \widetilde {F}^\#_{D}(s)u^{-s} \, ds,
\end{align*}
where $u>0$. We establish the following lemma.  
 \begin{lem} \label{TildeFchiSharp}
  Let $s\in\mathbb C$ with ${\rm Re}(s)>1$. The dual function  $\widetilde {F}^\#_{D}(s)$ satisfies the functional equation 
 \[
 \widetilde{F}^\#_{D}(s) = \widetilde{F}_{D}(-s)= \sqrt{D} \bigg(\frac{\sqrt{D}}{2\pi}\bigg)^{s} \Gamma(s) L(1+s,{\chi_{D}}) \zeta(s).
 \]
 \end{lem}
\begin{proof} We compute
\begin{align*}
 \widetilde {F}^\#_{D}(s) & = \sqrt{D}\int_0^\infty  \sum_{m=1}^\infty \sum_{n=1}^\infty \frac{{\chi_{D}}(m)}{m} e^{-\frac{2\pi  mnu}{\sqrt{D}}} \, u^s \,\frac{du}{u} = \sqrt{D} \sum_{m=1}^\infty \sum_{n=1}^\infty  \frac{{\chi_{D}}(m)}{m} \int_0^\infty   e^{-\frac{2\pi  mnu}{\sqrt{D}}} \, u^s \,\frac{du}{u}\\
  & = \sqrt{D} \sum_{m=1}^\infty \sum_{n=1}^\infty  \frac{{\chi_{D}}(m)}{m}\cdot \sqrt{D}^s\, (2\pi m n)^{-s} \Gamma(s) = \sqrt{D} \bigg(\frac{\sqrt{D}}{2\pi}\bigg)^{s} \Gamma(s) L(s+1,{\chi_{D}}) \zeta(s).
  \end{align*}
Combining this with Lemma \ref{FunctionalEquationTildeFchi} completes the proof.
  \end{proof}
 
\subsection{Proof of Theorem \ref{Thm : Modular Relation}} \label{ProofMainTheorem}
We now present the proof of the modular relation \eqref{Thm: Modular Relation}. Since both sides of \eqref{Thm: Modular Relation} are analytic functions on $\mathfrak{h}$, by the Identity Theorem, it suffices to prove \eqref{Thm: Modular Relation} in the special case when $z=iy$ for any $y >0$. It follows from Lemma \ref{FunctionalEquationTildeFchi}
that $\widetilde{F}_{D}(s)$ has a double pole at $s=0$ with residue $L'(0,{\chi_{D}})$ and a pole at $s=-1$ with residue $\frac{\pi \,L(-1,\,{\chi_D})}{\sqrt{D}}$ and no other poles. Fix $\sigma > 1.$ Using inverse Mellin transform and shifting the line of integration to the left, we have
  \begin{align*}
  F_{D}(iy)  &= \frac{1}{2\pi i} \int_{\sigma-i\infty}^{\sigma+i\infty} \widetilde {F}_{D}(s)y^{-s} \, ds \\
  &= L'(0,{\chi_{D}}) +y\cdot\frac{\pi \, L(-1,{\chi_{D}})}{\sqrt{D}} + \frac{1}{2\pi i} \int_{-\sigma-i\infty}^{-\sigma+i\infty} \widetilde {F}_{D}(s)y^{-s} \, ds\\
  & =  L'(0,{\chi_{D}}) +y\cdot\frac{\pi \, L(-1,{\chi_{D}})}{\sqrt{D}}+ \frac{1}{2\pi i} \int_{\sigma-i\infty}^{\sigma+i\infty} \widetilde {F}_{D}(-s)y^{s} \, ds.
  \end{align*} 
  Now we apply Lemma \ref{TildeFchiSharp}. It follows that
\begin{align}\label{Eq: Before Exponentiation}
  F_{D}(iy) 
  & =  L'(0,{\chi_{D}}) +y\cdot\frac{\pi \, L(-1,{\chi_{D}})}{\sqrt{D}} +    \frac{1}{2\pi i} \int_{\sigma-i\infty}^{\sigma+i\infty}\widetilde{F}^\#_{D}(s)y^{s}\, ds \notag\\
  & = L'(0,{\chi_{D}}) +y\cdot\frac{\pi \, L(-1,{\chi_{D}})}{\sqrt{D}} + F_{D}^\#(i/y). 
  \end{align} 
Finally, we exponentiate both sides of \eqref{Eq: Before Exponentiation}. Using the identities \eqref{Fchi} and \eqref{FchiSharp}, we obtain
   $$\Phi_{D}(iy)^{-1} = e^{L'(0,{\chi_{D}})+y\cdot\frac{\pi \, L(-1,{\chi_D})}{\sqrt{D}}}\cdot \Phi_{D}^\#(i/y)^{-1}$$
   which is equivalent to
   $$\Phi_{D}^\#(i/y) = e^{L'(0,{\chi_{D}})+y\cdot\frac{\pi \, L(-1,{\chi_D})}{\sqrt{D}}}\cdot \Phi_{D}(iy).$$
This completes the proof. \qed

\section{Modular Forms on $H(\sqrt{D})$: Proofs of Theorems \ref{Theorem: Delta5isModular} and \ref{Theorem: D>5}}\label{sec: Proof of Modularity Theorems}
\subsection{Preliminaries} Here we prove Theorems \ref{Theorem: Delta5isModular} and \ref{Theorem: D>5}. We need a preliminary lemma regarding values of quadratic Dirichlet $L$-functions at $s=-1$. Some cases of this result has been alluded to before (see Zagier \cite[Pg. 54]{Zagier}). We provide a complete proof for any $D \equiv 1 \bmod 4$ for the reader's convenience.
\begin{lem}\label{lem: Value at s=-1}
Let $D\equiv 1\bmod{4}$ with $D>0$ be a fundamental discriminant of a real quadratic field. Let $\chi_{D}$ be the primitive real character modulo $D$. Then 
\[
L(-1,{\chi_D}) = \begin{cases}
  -\frac{2}{5},  & \textrm{if } D=5 \\
  -2\mathfrak{m}_D, \quad \textrm{for some $\mathfrak{m}_D \in \mathbb{N}$}, & \textrm{if } D > 5.
\end{cases}
\]
\end{lem}
\begin{proof}
The case $D=5$ can be verified numerically. So we may assume $D>5$. We have
\begin{align*}
L\left(-1, {\chi_{D}}\right) & =D \sum_{n=1}^{D} {\chi_{D}}(n) \zeta\left(-1, n/D\right)  =-\frac{D}{2} \sum_{n=1}^{D} {\chi_{D}}(n)\left(\frac{n^2}{D^2}-\frac{n}{D}+\frac{1}{6}\right) =-\frac{1}{2 D} \sum_{n=1}^{D} n^2 {\chi_{D}}(n),
\end{align*}
where $\zeta(s, a)$ is the Hurwitz zeta function and since ${\chi_{D}}$ is even, we have noted that
\[
\sum_{n=1}^{D} {\chi_{D}}(n)=\sum_{n=1}^{D} n \cdot {\chi_{D}}(n)=0 .
\]
Writing
\[
S({\chi_{D}}) = \sum_{n=1}^{D}n^2 {\chi_{D}}(n),
\]
it suffices to show that $S({\chi_{D}})$ is divisible by $4D$. Since ${\chi_{D}}$ is primitive and real, $D$ must be squarefree. We consider the following cases.\\

\noindent \textbf{Case 1:}  $D$ is prime. We write
\[
S({\chi_{D}}) = \sum_{n=1}^{D}n^2 {\chi_{D}}(n) = \sum_{\substack{1 \leq n \leq D-1 \\ n \equiv \square \bmod D}}n^2-\sum_{\substack{1 \leq n \leq D-1 \\ n \not \equiv \square \bmod D}}n^2 =S_1({\chi_{D}})-S_2({\chi_{D}}). 
\]
Let $g$ be a primitive root of $D$. Then the non-residues are congruent to the odd powers of $g$. Therefore, their squares are congruent to $g^2, g^6, g^{10}$, and so on up to $g^{2q-4}$. Thus we have 
\begin{align}\label{Eq: Congruence}
\left(g^4-1\right) S_2({\chi_{D}}) \equiv g^2(g^{2(D-1)}-1) \bmod D.
\end{align}
The right hand side of \eqref{Eq: Congruence} is congruent to zero by Fermat's Little Theorem. Since $D>5$, we have $g^4-1 \not \equiv 0 \bmod D$. This implies that $S_2({\chi_{D}}) \equiv 0 \bmod D$. Since $S_1({\chi_{D}})+S_2({\chi_{D}})$ is the sum of squares up to $D-1$, we also have $S_1({\chi_{D}})+S_2({\chi_{D}}) \equiv 0 \bmod D$. It follows that $S({\chi_{D}}) \equiv 0 \bmod D$.
\par
To check divisibility by $4$, note that since $-1$ is a quadratic residue modulo $D$, $S_1({\chi_{D}})$ is odd if $D \equiv 5 \bmod 8$ and even if $D \equiv1 \bmod 8$. The claim is now immediate by writing
\[ S_1({\chi_{D}})-S_2({\chi_{D}}) = 2S_1({\chi_{D}}) - \frac{(D-1)D(2D-1)}{6},
\]
and carrying out a parity check.\\

\noindent \textbf{Case 2:}  $D$ is squarefree of the form $D = \prod_{i=1}^{k}p_i,$ for $p_i$ prime and $k \geq 2$. We only consider the case when $k=2$. The argument is similar for $k>2$.
\par
We write $D=p_1p_2$. By the Chinese Remainder Theorem, we can break $S({\chi_{D}})$ as
\begin{align}\label{Eq: CRT}
S({\chi_{D}}) &= \bigg ( \sum_{\substack{1 \leq n_1 \leq p_1 \\ n_1 \equiv \square \bmod p_1}} n_1^2 \bigg) \bigg( \sum_{\substack{1 \leq n_2 \leq p_2 \\ n_2 \equiv \square \bmod p_2}} n_2^2 \bigg) -  \bigg ( \sum_{\substack{1 \leq n_1 \leq p_1 \\ n_1 \not\equiv \square \bmod p_1}} n_1^2 \bigg) \bigg( \sum_{\substack{1 \leq n_2 \leq p_2 \\ n_2 \equiv \square \bmod p_2}} n_2^2 \bigg) \notag \\
&\quad - \bigg ( \sum_{\substack{1 \leq n_1 \leq p_1 \\ n_1 \equiv \square \bmod p_1}} n_1^2 \bigg) \bigg( \sum_{\substack{1 \leq n_2 \leq p_2 \\ n_2 \not \equiv \square \bmod p_2}} n_2^2 \bigg)+ \bigg ( \sum_{\substack{1 \leq n_1 \leq p_1 \\ n_1 \not \equiv \square \bmod p_1}} n_1^2 \bigg) \bigg( \sum_{\substack{1 \leq n_2 \leq p_2 \\ n_2 \not \equiv \square \bmod p_2}} n_2^2 \bigg).
\end{align} 
By Case 1, each term on the right hand side of \eqref{Eq: CRT} is divisible by both $p_1$ and $p_2$, and hence by $D$. It follows that $S({\chi_{D}}) \equiv 0 \bmod D$. As for divisibility by $4$, observe that by Case 1, the sum of the first two terms on the right-hand side of \eqref{Eq: CRT} is congruent to $0 \bmod 4$. Likewise, the sum of the last two terms of \eqref{Eq: CRT} is also congruent to $0 \bmod 4$.
\end{proof}

By \eqref{Defn: eta5}, \eqref{Defn: etaD} and Lemma \ref{lem: Value at s=-1}, we see that for $z \in \mathfrak{h}$,
\begin{align}\label{Eta_q definition rewritten}
\eta_{D}(z) =e^{-\frac{\pi i L(-1,{\chi_D})}{\sqrt{D}}\cdot z} \cdot \Phi_{D}(z) \Phi_{D}^\#(z).
\end{align} 
Using Theorem \ref{Thm : Modular Relation}, we have the relations
\begin{align}
    e^{-\frac{\pi i L(-1,{\chi_D})}{\sqrt{D}}\cdot z} \cdot  \Phi_{D}(z) &=e^{-L'(0,{\chi_{D}})} \cdot \Phi^\#_{D}(-1/z) \\
\textrm{and} \quad \Phi_{D}^\#(z) &=e^ {L'(0,{\chi_{D}})} \cdot e^{-\frac{\pi i L(-1,{\chi_D})}{\sqrt{D}}\cdot\frac{-1}{z}} \cdot \Phi_{D}(-1/z).
\end{align}
Multiplying the two relations above, we obtain
\begin{equation} \label{ModularRelation General} 
e^{-\frac{\pi i L(-1,{\chi_D})}{\sqrt{D}}\cdot z} \cdot \Phi_{D}(z) \Phi_{D}^\#(z)\; = e^{-\frac{\pi i L(-1,{\chi_D})}{\sqrt{D}}\cdot\frac{-1}{z}} 
 \cdot \Phi_{D}(-1/z) \Phi^\#_{D}(-1/z).  
\end{equation}
\subsection{Proof of Theorem \ref{Theorem: Delta5isModular}} Since $z \in \mathfrak{h}$, the products in \eqref{Defn: eta5} are convergent, and thus $\eta_5(z)$ is holomorphic over $\mathfrak{h}$. As mentioned in Remark \ref{Rem: eta5}, we will show that for every
$\gamma \in H\big(\sqrt{5}\big)$ we have $$\eta_5\left(\gamma z \right) = u_\gamma\cdot \eta_5(z)$$ where $u_\gamma$ is a certain fifth root of unity. In fact, we shall precisely compute $u_{\gamma}$.

\par
By \eqref{Eta_q definition rewritten} and \eqref{ModularRelation General}, we have
\begin{align}
    \eta_5(z+\sqrt{5}) &= e^{\frac{2\pi i}{5}} \eta_5(z) \label{Eq: Relation 1, q=5}\\
    \textrm{and} \quad \eta_5(-1/z) &= \eta_5(z). \label{Eq: Relation 2, q=5}
\end{align}
Choose any element of $H(\sqrt{5})$. It can be expressed in the form
\begin{equation*} 
\alpha(k_1,k_2, \ldots,k_\ell) := T_{\sqrt{5}}^{k_1} \cdot S \cdot T_{\sqrt{5}}^{k_2} \cdot  S  \cdots  T_{\sqrt{5}}^{k_{\ell-1}}\cdot S \cdot T_{\sqrt{5}}^{k_\ell}
\end{equation*}
with $k_1, k_2, \ldots,k_\ell\in\mathbb Z$, where
$$T_{\sqrt{5}}:= \begin{pmatrix} 1&\sqrt{5} \\0&1\end{pmatrix}\quad \textrm{and} \quad S = \begin{pmatrix} 0&-1\\1&\phantom{-}0\end{pmatrix}.$$
We consider the following two cases.\\

\noindent \textbf{Case 1 :} $\ell$ is even. In this case, it is straightforward to check that the matrix $\gamma$ is of the form
\[
\gamma = \begin{pmatrix} a\sqrt{5}&b\\
c  & d\,\sqrt{5}\end{pmatrix},
\]
where $a,b,c,d \in \mathbb{Z}$. We claim that 
\begin{align}\label{Gamma Claim}
u_{\gamma} = e^{\frac{2 \pi i c(a+d)}{5}}.
\end{align}
To prove this, we proceed by induction. First, by \eqref{Eq: Relation 1, q=5} and \eqref{Eq: Relation 2, q=5}, we have
\begin{equation} \label{ExplicitModularityRelation}
\eta_5\Big(\alpha(k_1,\ldots,k_\ell)\, z\big) = e^{{\frac{2\pi i(k_1+k_2+\cdots+k_\ell)}{5}} }\cdot \eta_5(z).
\end{equation}
For the base case, note that an element of the form $\alpha(k_1,\ldots,k_\ell)$ with $\ell$ even and the smallest  number of $S$-terms is $$\alpha(k_1,k_2) = T_{\sqrt{5}}^{k_1} \cdot S \cdot T_{\sqrt{5}}^{k_2} = \begin{pmatrix}k_1\sqrt{5}& -1+5k_1k_2\\
1& k_2  \sqrt{5} \end{pmatrix}.$$
One immediately checks using \eqref{ExplicitModularityRelation} that \eqref{Gamma Claim} holds in this case. We additionally note the congruence conditions $bc\equiv -1\bmod{5} $ and $c^2 \equiv 1 \bmod 5$, which we shall use in our induction step.

Let 
$$
\alpha(k_1,k_2, \ldots, k_{2n}) = \begin{pmatrix} a\sqrt{5}&b\\
c  & d\,\sqrt{5}\end{pmatrix}$$
where $a,b,c,d\in\mathbb Z$ and the congruence conditions $bc\equiv -1\bmod{5} $ and $c^2 \equiv 1 \bmod 5$ holds. Assume that \eqref{Gamma Claim} holds for $\alpha(k_1, k_2, \ldots, k_{2n})$. We now show that this assumption implies \eqref{Gamma Claim} holds for $\alpha(k_1, k_2, \ldots, k_{2n+2})$. Note that
$$S \cdot T_{\sqrt{5}}^{k_{2n+1}} \cdot S \cdot T_{\sqrt{5}}^{k_{2n+2}}=
\begin{pmatrix}-1 & -k_{2n+2}\,\sqrt{5}\\
k_{2n+1}\,\sqrt{5}&\;\; -1+5k_{2n+1}k_{2n+2} \end{pmatrix}.$$
Then
\begin{align*}
\alpha(k_1, k_2, \ldots, k_{2n+2}) &=  \begin{pmatrix} a\sqrt{5}&b\\
c & d\,\sqrt{5}\end{pmatrix} \cdot \begin{pmatrix}-1 & -k_{2n+2}\,\sqrt{5}\\
k_{2n+1}\,\sqrt{5}&\;\; -1+5k_{2n+1}k_{2n+2} \end{pmatrix}= \begin{pmatrix} a'\sqrt{5}&b'\\
c'  & d'\,\sqrt{5}\end{pmatrix},
\end{align*}
where we explicitly compute that
\begin{align*} & a' \equiv 4a + b  k_{2n+1}   \bmod{5}\\
& b'\equiv 4b    \bmod{5}\\
& c'\equiv 4c    \bmod{5}\\
& d' \equiv 4d+4c \,k_{2n+2} \bmod{5}.
\end{align*}
It follows (using the fact that $bc\equiv-1\bmod{5}$ and $c^2\equiv 1\bmod{5}$) that 
\begin{align*}
c'(a'+d') &\equiv ac + cd +4bc\,k_{2n+1} + c^2\,k_{2n+2}  \bmod{5}\\
& \equiv c(a+d) + k_{2n+1} + k_{2n+2}  \bmod{5}.
\end{align*}
It follows using \eqref{ExplicitModularityRelation} that \eqref{Gamma Claim} holds for $\alpha(k_1, k_2, \ldots, k_{2n+2})$. The desired claim follows.\\

\noindent \textbf{Case 2 :} $\ell$ is odd. In this case, the matrix $\gamma$ is of the form
\[
\gamma = \begin{pmatrix} a&b\sqrt{5}\\
c\sqrt{5}  & d\,\end{pmatrix},
\]
where $a,b,c,d \in \mathbb{Z}$. Then we have $$
S\cdot \gamma = \begin{pmatrix}   0&-1\\1&0
\end{pmatrix} \begin{pmatrix} a&b\,\sqrt{5}\\
c \,\sqrt{5} & d\end{pmatrix} = \begin{pmatrix} -c\,\sqrt{5}&-d\\
a  & b\,\sqrt{5}\end{pmatrix}.$$
By \eqref{Eq: Relation 2, q=5} and Case 1, we have
\[
\eta_5(\gamma z) = \eta_5(S\gamma z) = u_{S\gamma} \cdot \eta_5(z) =e^{\frac{2\pi i a(b-c)}{5}}  \eta_5(z).
\]
Therefore in this case, one obtains
\begin{align}\label{Gamma Claim 2}
u_{\gamma} = e^{\frac{2 \pi i a(b-c)}{5}}.
\end{align}
\par
Combining \eqref{Gamma Claim} and \eqref{Gamma Claim 2}, we have for any $\gamma \in H(\sqrt{5})$,
\[
\eta_5(\gamma z) = u_{\gamma} \cdot \eta_5(z), 
\]
where $u_{\gamma}$ is a certain fifth root of unity. Consequently, \( \Delta_5(z) = \eta_5(z)^5 \) is invariant under the action of \( H(\sqrt{5}) \), which establishes the modularity of \( \Delta_5 \) as stated in Theorem~\ref{Theorem: Delta5isModular}. Moreover from the definition in \eqref{Defn: eta5}, it follows that \( \Delta_5(z) \) vanishes at the cusp $\infty$, that is, $\lim_{y \to \infty} \Delta_5(iy) = 0.$ \qed

\subsection{Proof of Theorem \ref{Theorem: D>5}: The Case $D>5$}
Again since $z \in \mathfrak{h}$, the products in \eqref{Defn: etaD} are convergent, and thus $\eta_{D}(z)$ is holomorphic over $\mathfrak{h}$. For $D>5$, by \eqref{Eta_q definition rewritten} and Lemma \ref{lem: Value at s=-1}, we have
\begin{align}
\eta_{D}(z+\sqrt{D}) =e^{-\pi i L(-1,{\chi_{D}})} \cdot \eta_{D}(z) = \eta_{D}(z).
\end{align}
From \eqref{ModularRelation General}, we immediately obtain the relations
\begin{align}
    \eta_{D}(z+\sqrt{D}) &= \eta_{D}(z) \label{Eq: Relation 1, q>5}\\
    \textrm{and} \quad \eta_{D}(-1/z) &= \eta_{D}(z). \label{Eq: Relation 2, q>5}
\end{align}
Since $H(\sqrt{D})$ generated by 
$$T_{\sqrt{D}}= \begin{pmatrix}1&\sqrt{D}\\0&1\end{pmatrix}, \quad \textrm{and} \quad S = \begin{pmatrix}0&-1\\1&0\end{pmatrix},$$
we see that \( \eta_{D}(z) \) is invariant under the action of \( H(\sqrt{D}) \). This establishes the modularity of \( \eta_{D} \) as stated in Theorem~\ref{Theorem: D>5}. Moreover from the definition in \eqref{Defn: etaD}, it follows that \( \eta_{D}(z) \) vanishes at the cusp $\infty$, that is, $\lim_{y \to \infty} \eta_{D}(iy) = 0.$\qed

\section{Asymptotics for Fourier Coefficients: Part I} \label{Sec: Main Theorem Proof I} In this section, we establish some preliminary lemmas towards the proof of Theorem \ref{gen thm}. We aim to apply Cauchy's formula in the form
\begin{align}\label{Eq: Cauchy Formula}
a_D(N)=\frac{1}{2\pi i}\int_{|q|=r}G(q)\, \frac{dq}{q^{N+1}}
\end{align}
for some $0<r<1$ where
\begin{align}\label{Eq: Product Formula}
G(q) = \prod_{n=1}^\infty \bigg(\left(1-q^n   \right)^{{\chi_{D}}(n)}  \prod_{a=1}^D \left(1-  \zeta_D^a q^n  \right)^{{\chi_{D}}(a)}\bigg), \quad  q = e^{\frac{2\pi i z}{\sqrt{D}}},
\end{align}
and isolate the integral in \eqref{Eq: Cauchy Formula} around the dominant singularities of $G(q)$. Our goal is to find the location of these singularities and give a precise description of $G(q)$ around these points.   
  
\begin{lem}\label{lem 1} For $|q|<1$, we have 
\begin{equation}\label{log F}
\log G(q)
=
-\sum_{m\geq 1}\frac{g(q^m)}{m(1-q^{Dm})}
-\sqrt{D}\sum_{m\geq 1}\frac{\chi_D(m)}{m}\frac{q^m}{1-q^m}
\end{equation}
where $
g(x)=\sum_{a=1}^{D-1}\chi_D(a)x^a$
is the Fekete polynomial corresponding to the character $\chi_D$.
\end{lem}
\begin{proof} From the product formula \eqref{Eq: Product Formula}, we have 
  \begin{align*}
 \log G(q)
&=
  \sum_{n\geq 1} \bigg(\chi_D(n)\log(1-q^n)+\sum_{a=1}^D \chi_D(a)\log(1-\zeta_D^a q^n)\bigg)
  \\
&=
    -\sum_{n\geq 1} \bigg(\chi_D(n)\sum_{m\geq 1}\frac{q^{nm}}{m}+\sum_{a=1}^D \chi_D(a)\sum_{m\geq 1}\frac{\zeta_D^{am}q^{nm}}{m}\bigg).
  \end{align*}
Since $\chi_D$ is a primitive real character, $\sum_{a=1}^D \chi(a)\zeta_D^{am}=\chi_D(m)\tau(\chi_D)$ for all $m$ where $\tau(\chi_D)$ is the Gauss sum. Thus the above is 
   \begin{equation}\label{log prod}
    -\sum_{n\geq 1} \bigg(\chi_D(n)\sum_{m\geq 1}\frac{q^{nm}}{m}+\tau(\chi_D)\sum_{m\geq 1}\frac{\chi_D(m)q^{nm}}{m}\bigg).
     \end{equation} 
Since $\tau(\chi_D)=\sqrt{D}$, evaluating the sum over $n$ gives us the desired result.
\end{proof}
Recall that in the statement of Theorem \ref{gen thm}, we set $c_D =\sqrt{D}L(2,\chi_D)/\mf{n}_D^2$. The dominant singularity of $G(q)$ will roughly be of size $\exp(c_D/(1-|q|))$. We now give a minor arc estimate demonstrating that when $q$ is sufficiently far away from a low root of unity, $G(q)$ is substantially smaller than this. 

\begin{lem}\label{crude bound lem}
%Let 
%\[
%c_D=\frac{\sqrt{D}L(2,\chi_D)}{\mf{n}_D^2}
%\]
%where $\mf{n}_D$ is the least quadratic residue modulo $D$. Then
For $A>0$ sufficiently large in terms of $D$, we have  
\begin{align}\label{Eq: Low Roots}
\bigg|\log G(q)+\sum_{m\leq A}\frac{g(q^m)}{m(1-q^{Dm})}
+\tau(\chi_D)\sum_{m\leq A}\frac{\chi_D(m)}{m}\frac{q^m}{1-q^m}\bigg|\leq \frac14\cdot\frac{c_D}{(1-|q|)}.
\end{align}
In particular, there exists $B>0$ depending at most on $A$ and $D$ such that if $q$ satisfies $|1-q^m|\geq B(1-|q|)$ for all $1 \leq m \leq AD$, then
\begin{align}\label{Eq: Low Roots II}
|\log G(q)|\leq\frac12\cdot\frac{c_D}{(1-|q|)}. 
\end{align}
\end{lem}
\begin{proof}
The left hand side of \eqref{Eq: Low Roots} is less than 
\begin{align}\label{Eq: Truncation}
    \bigg \lvert \sum_{m> A}\frac{g(q^m)}{m(1-q^{Dm})}
+\tau(\chi_D)\sum_{m > A}\frac{\chi_D(m)}{m}\frac{q^m}{1-q^m}\bigg \rvert \leq (D+\sqrt{D})\sum_{m> A}\frac{|q|^m}{m(1-|q|^m)}.
\end{align}
Since 
\begin{align*}
1-|q|^m
= &
(1-|q|)(1+|q|+\cdots+|q|^{m-1})
\geq 
(1-|q|)m |q|^m
\end{align*}
for $|q|\leq 1$, \eqref{Eq: Truncation} is
\[
\leq \frac{2D}{1-|q|}\sum_{m>A} \frac{1}{m^2}< \frac{ 2D}{1-|q|}\int_A^\infty \frac{dx}{x^2}
=
\frac{2D}{A(1-|q|)}.
\]
Taking $A$ sufficiently large depending on $D$, the bound in \eqref{Eq: Low Roots} follows. On taking $B$ sufficiently large depending on $A$ and $D$, the bound in \eqref{Eq: Low Roots II} follows immediately by the triangle inequality. 
\end{proof}

To set ideas, we note that we will eventually take $|q|=r\sim1-\sqrt{c_D/N}$ in \eqref{Eq: Cauchy Formula}. For such $q$, the condition $|1-q^m|\leq B(1-|q|)$ then necessitates that $\arg q=2\pi u/v+\theta$ for some $1 \leq u \leq v\leq m$ and $|\theta|\ll 1/\sqrt{N}\ll 1-r$. 
Our next proposition gives a more precise description of $G(q)$ around these remaining low roots of unity. 
Our results involve the periodic zeta function defined for $\alpha\in\mb{R}$ by
\[
P(s,\alpha)=\sum_{n\geq 1}\frac{e(\alpha n)}{n^s}
\]
when $\operatorname{Re}(s)>1$ (or $\operatorname{Re}(s)>0$ when $\alpha\in\mb{R}\backslash\mb{Z}$) and extended to all of $\mb{C}$ by analytic continuation. If $\alpha\in\mb{Z},$ we have $P(s,\alpha)=\zeta(s)$; otherwise it is entire. In either case, it is of at most polynomial growth in fixed vertical strips. 
\begin{prop}\label{precise lem}
Let $\varepsilon>0$. Fix $u,v \in \mathbb{N}$ satisfying $1 \leq u \leq v$ and $(u,v)=1$. Set $q=re^{2\pi i {u}/{v} +i\theta}$ where $0<r<1$ and $|\theta|\leq B(1-r)$ for some fixed constant $B>0$. If $D\nmid v$, we have 
\begin{equation}\label{log F 2}
\log G(q)
=
%\frac{\sqrt{5}L(2,\chi_D)}{4(\log(1/r)-i\theta)}
-\frac{\chi_D(v)\sqrt{D}L(2,\chi_D)}{v^2(\log(1/r)-i\theta)}
+\gamma_D(u,v)
+\mathcal{O}_{D,B,\varepsilon}(v^{1+\varepsilon}(1-r))
\end{equation}
where 
\begin{align}\label{Eq: Gamma Def}
\gamma_D(u,v)
&=
\chi_D(v) L(1,\chi_D) \frac{(1-v)\sqrt{D}}{2v}
\notag \\
&\quad -
\frac{1}{\phi(v)}\sum_{\psi \bmod v} \sum_{a=1}^{v-1}\overline{\psi}(a)\Big[\sqrt{D}P(0,au/v)L(1,\chi_D\psi)
+
P(1,au/v)L(0,\chi_D\psi)\Big].
\end{align}
%and $P(s,\alpha)$ is the  periodic zeta function. 

If $D|v$, we have 
\begin{equation}\label{log F 3}
\log G(q)
=
%\frac{\sqrt{5}L(2,\chi_D)}{4(\log(1/r)-i\theta)}
-\frac{\chi_D(u)D^{3/2}L(2,\chi_D)}{v^2(\log(1/r)-i\theta)}
+a_D(u,v)\log(\log(1/r)-i\theta)+b_D(u,v)+\mathcal{O}_{D,B,\varepsilon}(v^{1+\varepsilon}(1-r))
\end{equation}
for some computable constants $a_D(u,v),b_D(u,v)$ depending only on $u,v$ and $D$. 
\end{prop}

\begin{proof}
For notational simplicity, let us replace $q$ by $qe(u/v)$ where we set $q=re^{i\theta}$ and $|\theta|\leq B(1-r)$. Recall from the proof of Lemma \ref{lem 1} that we may write
\begin{equation}\label{log F twist}
\log G(qe(u/v))=-S_1-\sqrt{D}S_2
\end{equation}
with 
\[
S_1
=
\sum_{m,n\geq 1}\frac{\chi_D(n)e(mnu/v)q^{mn}}{m}
\qquad \textrm{and} \qquad
S_2
=
\sum_{m,n\geq 1}\frac{\chi_D(m)e(mnu/v)q^{mn}}{m}.
\]
We break our proof into the following parts.\\

 \noindent \textbf{$S_2$ Analysis.} We look at $S_2$ first and consider the sum over $m$ modulo $v$. Write $S_2=S_{21}+S_{22}$ with 
 \[
 S_{21}
 =
 \frac{\chi_D(v)}{v}\sum_{m,n\geq 1}\frac{\chi_D(m)}{m}q^{vmn},
 \qquad \textrm{and} \qquad
 S_{22}
 =
 \sum_{a=1}^{v-1}\sum_{\substack{m,n\geq 1 \\ m\equiv a \bmod v}}
 \frac{\chi_D(m)e(anu/v)q^{mn}}{m}.
 \]

Applying the identity
\begin{equation}\label{mellin}
\qquad e^{-z}=\frac{1}{2\pi i }\int_{c-i\infty}^{c+i\infty}\Gamma(s)z^{-s}\, ds, \qquad c>0,
\end{equation}
which is valid for $\operatorname{Re}(z)>0$, we have
\begin{align*}
S_{21} &=
\frac{\chi_D(v)}{v}
\sum_{m,n\geq 1}\frac{\chi_D(m)}{m}
\frac{1}{2\pi i }\int_{2-i\infty}^{2+i\infty}\Gamma(s)(vmn\log(1/q))^{-s} \, ds
\\
&=
\frac{\chi_D(v)}{v}
\frac{1}{2\pi i }\int_{2-i\infty}^{2+i\infty}
\Gamma(s)\zeta(s)L(1+s,\chi_D)(v\log(1/q))^{-s}\, ds.
\end{align*}
The integrand has simple poles at $s=1,0,-1$ and so we would like to shift the integral past these points. 

When estimating the integral in the left half-plane note that since $|\theta|\leq B(1-r)$, 
\[
|\arg \log(1/q)|=\arctan(|\theta|/\log(1/r))\leq \tfrac{\pi}{2}-\delta
\] 
for some fixed $\delta=\delta(B)\asymp \min (1,B^{-1})$. In particular, by Stirling's formula we have 
\[
\Gamma(s)(\log(1/q))^{-s}\ll |\log(1/q)|^{-\operatorname{Re}(s)} e^{-\delta \operatorname{Im}(s)/2}, \quad \textrm{as} \quad \operatorname{Im}(s)\to\infty.
\]
Hence the integral over the new line is absolutely convergent and $\ll_{D,B} |\log(1/q)|^{-\operatorname{Re}(s)}$ since $\zeta(s)$ and $L(1+s,\chi_D)$  are of polynomial growth at most.
Thus shifting to $\operatorname{Re}(s)=-3/2$, say, we find that 
\begin{align*}
S_{21}
= &
\frac{\chi_D(v)L(2,\chi_D)}{v^2\log(1/q)}
-\frac{\chi_D(v)}{2v}
L(1,\chi_D)
+\mathcal{O}_{D,B}(v^{1/2}|1-q|)
\end{align*}
on recalling that $\zeta(0)=-1/2$.

By the orthogonality of Dirichlet characters, 
\[
S_{22}
=
\frac{1}{\phi(v)}\sum_{\psi \bmod v} \sum_{a=1}^{v-1}\overline{\psi}(a)\sum_{\substack{m,n\geq 1}}
 \frac{\chi_D(m)\psi(m)e(anu/v)q^{mn}}{m}.
\]
On applying \eqref{mellin}, this is equal to
\begin{align}\label{Eq: Double Pole}
\frac{1}{\phi(v)}\sum_{\psi \bmod v} \sum_{a=1}^{v-1}\overline{\psi}(a)
\frac{1}{2\pi i }\int_{2-i\infty}^{2+i\infty}
\Gamma(s)P(s,au/v)L(1+s,\chi_D\psi)(\log(1/q))^{-s}\, ds
\end{align}
where $P(s,\alpha)$ is the periodic zeta function. Since $au/v\in \mb{Q}\backslash\mb{Z}$, on shifting to the line $\operatorname{Re}(s)=-3/2$, we only encounter poles at $s=0,-1$. If $\chi_D\psi$ is principal, we have a double pole at $s=0$; otherwise it is simple. Note that if $\chi_D\psi$ is principal, we must have $D \mid v$. 
\begin{comment}
    Let us examine when this can occur. Writing $D=p_1\cdots p_r$ we see that $\chi_D=\chi_{p_1}\cdots \chi_{p_r}$ with $\chi_{p_i}$ the real character modulo $p_i$. Hence for $\chi_D\psi$ to be principal we must have $\psi=\chi_D\psi'$ with $\psi'$ principal and this can only occur if $D|v$. In this case  $\psi'$ is the principal character modulo $v/D$, or equivalently, $\psi=\chi_D\psi_0$ where $\psi_0$ is the prinicpal character modulo $v$, that is, $\psi$ is induced by $\chi_D$ and in this case $\chi_D\psi$ is the principal character modulo $v$. 
\end{comment}
Distinguishing between the two cases and accounting for the double pole when $D|v$, we have
\begin{align*}
S_{22}& =
\frac{\mathbbm{1}_{D\nmid v}}{\phi(v)}\sum_{\psi \bmod v} \sum_{a=1}^{v-1}\overline{\psi}(a)P(0,au/v)L(1,\chi_D\psi)
\\
&\quad +
\mathbbm{1}_{D| v}(a_D(u,v)\log\log(1/q)+b_D(u,v))
+
\mathcal{O}_{D,B}(v|1-q|),
\end{align*}
for some computable constants $a_D(u,v),b_D(u,v)$ depending only on $u,v$ and $D$. In total, we  find that 
\begin{align}\label{S2}
S_2
&=
\mathbbm{1}_{D\nmid v} \bigg(\frac{\chi_D(v)L(2,\chi_D)}{v^2\log(1/q)}
-\frac{\chi_D(v)}{2v}
L(1,\chi_D)
+
h_D(u,v)\bigg) \notag\\
&\quad +
\mathbbm{1}_{D| v}(a_D(u,v)\log\log(1/q)+b_D(u,v))
+\mathcal{O}_B(v|1-q|),
\end{align}
where 
\begin{align*}
h_D(u,v)
= &
\frac{1}{\phi(v)}\sum_{\psi \bmod v} \sum_{a=1}^{v-1}\overline{\psi}(a)P(0,au/v)L(1,\chi_D\psi).
%=
%\mathds{1}_{D\nmid v} \sum_{a=1}^{v-1}P(0,au/v)L(1,\chi_D, a)
\end{align*}
%and $L(s,\chi,a)=\sum_{n\equiv a\mod v}\chi(n)n^{-s}$.

\noindent \textbf{$S_1$ Analysis : The Case $D \nmid v$.} We consider the sum over $n$ modulo $v$ in $S_1$ and write $S_1=S_{11}+S_{12}$ with 
\[
S_{11}
=
\chi_D(v)\sum_{m,n\geq 1}\frac{\chi_D(n)q^{vmn}}{m}
,\qquad \textrm{and} \qquad
S_{12}
=
\sum_{a=1}^{v-1}\sum_{\substack{m,n\geq 1\\n\equiv a \bmod v}}\frac{\chi_D(n)e(amu/v)q^{mn}}{m}.
\] 
By \eqref{mellin}, we have
\[
S_{11}
=
\frac{\chi_D(v)}{2\pi i }\int_{2-i\infty}^{2+i\infty}
\Gamma(s)\zeta(1+s)L(s,\chi_D)(v\log(1/q))^{-s}\, ds.
\]
By the functional equation \eqref{FunctionalEquationForL(s,chi)}, $L(0,\chi_D)=0$ since $\chi_D$ is even. Thus the pole of the integrand at $s=0$ is simple. Shifting to $\operatorname{Re}(s)=-1-\varepsilon$ say, we see that
\begin{align}
 S_{11} & = \chi_D(v)L'(0,\chi_D)+\mathcal{O}_{B, \varepsilon}(v^{1+\varepsilon}|1-q|) \notag \\
 &=
  \frac{ \chi_D(v)\sqrt{D}}{2}L(1,\chi_D)+\mathcal{O}_{D,B, \varepsilon}(v^{1+\varepsilon}|1-q|), \label{S11}
 \end{align}
after another application of the functional equation.

For $S_{12}$, by orthogonality of Dirichlet characters, we have 
\begin{align*}
S_{12}& =
\frac{1}{\phi(v)}\sum_{\psi \bmod v}\sum_{a=1}^{v-1}\ol{\psi}(a)
\sum_{\substack{m,n\geq 1}}\frac{\chi_D(n)\psi(n)e(amu/v)q^{mn}}{m}
\\
& =
\frac{1}{\phi(v)}\sum_{\psi \bmod v}\sum_{a=1}^{v-1} 
\frac{\ol{\psi}(a)}{2\pi i}\int_{2-i\infty}^{2+i\infty}
\Gamma(s)P(1+s,au/v)L(s,\chi_D\psi)(\log(1/q))^{-s}\, ds.
\end{align*}
As noted before, the character $\chi_D\psi$ can only be principal if $D|v$ which we are currently assuming is not the case. Hence we only have simple poles at $s=0,-1$ and on shifting contours,  
we see
\begin{equation}\label{S12}
S_{12}
=
%\mathds{1}_{D|v}\bigg[\frac{\prod_{p|v}(1-p^{-1})}{\phi(v)\log(1/q)}\sum_{a=1}^{v-1} \chi_D\psi_0(a)P(2,au/v)+c_{u,v}\bigg]
%\\
%+
\frac{\mathbbm{1}_{D\nmid v}}{\phi(v)}\sum_{\psi \bmod v}\sum_{a=1}^{v-1} 
{\ol{\psi}(a)}P(1,au/v)L(0,\chi_D\psi)
+ \mathcal{O}_{D,B}(v|1-q|).
\end{equation}
%Combining \eqref{S2},\eqref{S11} and \eqref{S12} in \eqref{log F twist} gives the result for $D\nmid v$. 

%Now, of course $\phi(v)^{-1}\prod_{p|v}(1-p^{-1})=v^{-1}$ whilst
%\[
%\sum_{a=1}^{v-1} \chi_D\psi_0(a)P(2,au/v)
%=
%\sum_{n\geq 1}\frac{1}{n^2}\sum_{a=1}^{v-1}\chi_D\psi_0(a)e(aun/v).
%\]
%Since $(u,v)=1$ the inner sum on the right is 
%\[
%\chi_D(u)\sum_{a=1}^{v-1}\chi_D\psi_0(a)e(an/v).
%\]
%If $(n,v)=1$ we may also factor out $\chi_D\psi_0(n)$ and remove the $n$ in the exponent of the summand. If $g:=(n,v)>1$ then $n/v=c/d$, say, with $(c,d)=1$ and 

\noindent \textbf{$S_1$ Analysis : The Case $D \mid v$.} Here we use a different decomposition of the sum. Let $v=Dv'$. We write
\[
S_1
=
\sum_{a=1}^{D-1}\chi_D(a)\sum_{\substack{m,n\geq 1\\n\equiv a \bmod D}}\frac{e(umn/v)q^{mn}}{m}
=
\sum_{a=1}^{D-1}\chi_D(a)\sum_{\substack{n\geq 0,m\geq 1}}\frac{e(aum/v)e(umn/v')q^{m(a+Dn)}}{m}.
\]
Now we consider the sum over $m$ modulo $v'$ giving 
\[
S_1=S_{11}'+S_{12}'
\]
with 
\[
S_{11}'
=
\frac{1}{v'}\sum_{a=1}^{D-1}\chi_D(a)\sum_{\substack{m\geq 1}}\frac{e(aum/D)}{m}\sum_{n\geq 0}q^{v'm(a+Dn)}
\]
and
\[
S_{12}'
=
\sum_{a=1}^{D-1}\chi_D(a)\sum_{b=1}^{v'-1}\sum_{\substack{m\geq 1\\m\equiv b \bmod v'}}\frac{e(aum/v)}{m}\sum_{n\geq 0}e(ubn/v')q^{m(a+Dn)}.
\]
Applying Mellin inversion and writing $v'(a+Dn)=v(n+a/D)$, we have 
\[
S_{11}'
=
\frac{1}{v'}\sum_{a=1}^{D-1}\chi_D(a)
\frac{1}{2\pi i }\int_{2-i\infty}^{2+i\infty}
\Gamma(s)P(1+s, au/D)\zeta(s,a/D)(v\log(1/q))^{-s}\, ds
\]
where $\zeta(s,\alpha)=\sum_{n\geq 0}(n+\alpha)^{-s}$ is the Hurwitz zeta function. This can be continued to a meromorphic function on $\mb{C}$ with only a simple pole at $s=1$ with residue $1$, and with at most polynomial growth in fixed vertical strips. Since $au/D$ is not integral, the periodic zeta function in the integrand is entire and so the poles at $s=1,0,-1$ are all simple. Shifting contours to $\operatorname{Re}(s)=-1-\varepsilon$ gives
\begin{align} \label{S11' Contribution}
S_{11}'
&=
\frac{D}{v^2\log(1/q)}\sum_{a=1}^D\chi_D(a)P(2,au/D) \notag \\
&\quad +
\sum_{a=1}^D\chi_D(a)P(1,au/D)\zeta(0,a/D)
+\mathcal{O}_{D,B, \varepsilon}(v^{\varepsilon}|1-q|)
\end{align} 
since $vv'=v^2/D$.
The constant in the leading term here, that is, the sum over $a$ involving $P(2,au/D)$ can be evaluated using Gauss sums as follows: 
\begin{align*}
\sum_{a=1}^D\chi_D(a)P(2,au/D) &=\sum_{n\geq 1}\frac{1}{n^2}\sum_{a=1}^D\chi_D(a)e(aun/D)\\
&=
\tau(\chi_D)\chi_D(u)\sum_{n\geq 1}\frac{\chi_D(n)}{n^2}
=
\sqrt{D}\chi_D(u)L(2,\chi_D).
\end{align*}

For $S_{12}'$, \eqref{mellin} gives 
\[
\sum_{a=1}^{D-1}\chi_D(a)\sum_{b=1}^{v'-1}
\frac{1}{2\pi i }\int_{2-i\infty}^{2+i\infty}
\Gamma(s)\bigg(\sum_{\substack{m\geq 1\\m\equiv b \bmod v'}}\frac{e(aum/v)}{m^{1+s}}\bigg)\bigg(\sum_{n\geq 0}\frac{e(ubn/v')}{(n+a/D)^s}\bigg)(D\log(1/q))^{-s}\, ds.
\]
The Dirichlet series in the integrand here are given by Lerch zeta functions. Since neither $au/v$ or $ub/v'$ are integral, these are entire functions and we only pick up poles at $s=0,-1$. Therefore, we obtain
\begin{align}
\label{S12' Contribution}
S_{12}' = c_D(u,v) +\mathcal{O}_{D,B, \varepsilon}(v^{\varepsilon} \lvert 1-q \rvert),
\end{align}
for some computable constant $c_D(u,v)$ depending only on $u,v$ and $D$.\\

\noindent \textbf{Conclusion.} We now put together the cases $D|v$ and $D\nmid v$ for $S_1$. Combining \eqref{S11}--\eqref{S12' Contribution}, we have
\begin{align}\label{S1 Final}
S_1
&=
  \mathbbm{1}_{D\nmid v} \bigg(\frac{ \chi_D(v)\sqrt{D}}{2}L(1,\chi_D)+h_D'(u,v)\bigg)
\notag \\
&\quad +
\mathbbm{1}_{D|v}\bigg(\frac{\chi_D(u)D^{3/2}L(2,\chi_D)}{v^2\log(1/q)}+c_D'(u,v)\bigg)+\mathcal{O}_{D,B, \varepsilon}(v^{1+\varepsilon}|1-q|)
\end{align}
for some computable constant $c_D'(u,v)$ where 
\[
h_D'(u,v)
=
\frac{1}{\phi(v)}\sum_{\psi \bmod v}\sum_{a=1}^{v-1} 
{\ol{\psi}(a)}P(1,au/v)L(0,\chi_D\psi).
\] 
For clarity, we mention that the second term on the right hand side of \eqref{S11' Contribution} has been absorbed into $c_D'(u,v)$. Finally, we note that in our range of $\theta$, $\lvert 1-q \rvert \asymp_B 1-r.$ Therefore, combining \eqref{S1 Final} with \eqref{S2} in \eqref{log F twist} and readjusting the constants $a_D(u,v),b_D(u,v)$ accordingly gives the desired result.
\end{proof}

We pause briefly to detail the most dominant singularities. Observe that the constant factors in the leading terms of \eqref{log F 2} and \eqref{log F 3} can be made positive by choosing $u,v$ such that $\chi_D(v)$ or $\chi_D(u)=-1$, respectively. In the case of \eqref{log F 2}, the constant is largest when $v$ is smallest.  In other words, for those $v$ not divisible by $D$, the dominant singularity occurs at the points where $q\to e(u/v)$ and $v$ is the least quadratic non-residue modulo $D$, that is $v=\mf{n}_D$. In this case, we have 
\[
G(q)=\exp\bigg((1+o(1))\frac{\sqrt{D}L(2,\chi_D)}{\mf{n}_D^2(1-e(-u/v)q)}\bigg),\qquad \textrm{as} \quad q\to e(u/v),\,\,\, D\nmid v 
\] 
for all $1\leq u< v, (u,v)=1$. We have noted here that
\[
\log (1/r)-i \theta \sim 1-r e(\theta) =1-e(-u/v)q, \quad \textrm{as} \quad r \to 1 \quad \textrm{and} \quad \theta \to 0. 
\]
For those $v$ divisible by $D$, a similar heuristic shows that  
\[   G(q)=\exp\bigg((1+o(1))\frac{D^{3/2}L(2,\chi_D)}{v^2(1-e(-u/v)q)}\bigg),\qquad \textrm{as} \quad q\to e(u/v),\,\,\, D|v  
\] 
at points for which $\chi_D(u)=-1$. Clearly, the constant here is largest when $v=D$ and thus we see that the singularity for $D\nmid v$ is dominant provided $\mf{n}_D<\sqrt{D}$ which as mentioned in Remark \ref{Rem: Least Quadratic Nonresidue}, we believe to be always the case. All other singularities are sub-dominant, in fact, they are exponentially smaller.

Our final lemma in this section allows us to localize the variable $\theta$ in Proposition \ref{precise lem} 
 yet further (as well as set some notation). This is necessary for an application of the method of stationary phase in Section \ref{Sec: Main Theorem Proof II}.  

\begin{lem}\label{local lem}
Let $0<r<1$, $\theta\in(-\pi,\pi ]$, $N \geq 1$ be fixed and
\[
f(\theta)=\frac{c_D}{ i(\log(1/r)-i\theta)}-N\theta.
\] 
 Then
 \[
\frac{|e^{if(\theta)}|}{e^{c_D/\log(1/r)}}
\leq
\begin{cases}
e^{-\tfrac12c_D\theta^2/\log^3(1/r)},\qquad &|\theta|<\log(1/r)
\\
e^{-\tfrac12c_D/\log(1/r)},\qquad &|\theta|\geq \log(1/r).
\end{cases} 
\]
\end{lem}
\begin{proof}We have 
\begin{align*}
\operatorname{Re}\bigg(\frac{1}{\log(1/r)-{i\theta}}\bigg)-\frac{1}{\log(1/r)}
= &
\frac{\log(1/r)}{\log^2(1/r)+\theta^2}-\frac{1}{\log(1/r)}
\\
= &
-\frac{\theta^2}{\log(1/r)(\log^2(1/r)+\theta^2)}.
\end{align*}
The result then follows on distinguishing the size of $|\theta|$.
\end{proof}

\section{Asymptotics for Fourier Coefficients: Part II} \label{Sec: Main Theorem Proof II}
Our main objective in this section is to prove Theorem \ref{gen thm}.
\subsection{Initial preparations}
Let $N \geq 1$ and $q=re^{i\theta}$ with $\theta\in[0,2\pi)$. We set $r$ such that 
\[
\log(1/r)=\sqrt{\frac{c_D}{N}}
\]  
where we recall that $c_D=\sqrt{D}L(2,\chi_D)/\mf{n}_D^2$.
Note that $r=1-\sqrt{{c_D}/{{N}}}+\mathcal{O}_D(1/N)$ which we may use interchangeably. 

By Cauchy's formula, we have 
\begin{align*}
a_D(N)
= &
\frac{1}{2\pi i }\int_{|q|=r} G(q)\frac{dq}{q^{N+1}}
=
\frac{1}{2\pi r^N}\int_0^{2\pi}G(re^{i\theta})e^{-iN\theta}\, d\theta.
\end{align*}
Assuming that $\mf{n}_D<\sqrt{D}$, on applying Lemma \ref{crude bound lem} and Proposition \ref{precise lem}, along with the remarks after it, we see that the dominant singularity of $G$ occurs around the points $e(u/\mf{n}_D)$, $1\leq u<\mf{n}_D$, $(u,\mf{n}_D)=1$. More precisely, for $B$ sufficiently large in terms of $D$, we have
\begin{align*}
\frac{1}{2\pi r^N}\int_{|\theta-2\pi u/\mf{n}_D|>B/\sqrt{N}}G(re^{i\theta})e^{-iN\theta} \, d\theta
&\ll_D
 \exp\bigg(\frac{c_D-\varepsilon}{\log(1/r)}-N\log r\bigg)
 \\
&\ll_D \exp\bigg((1+o(1))(2-\varepsilon){ \sqrt{c_DN}}\bigg),
\end{align*}
for some fixed $\varepsilon>0$ depending only on $D$. This is exponentially smaller than our target bound $\exp(2\sqrt{c_DN})$.
 
 By the asymptotic \eqref{log F 2} in Proposition \ref{precise lem}, the remaining integrals are 
\begin{align}\label{Eq: Remaining Integral}
\frac{1}{2\pi r^N}&\sum_{\substack{u \bmod \mf{n}_D \\ (u,\mf{n}_D)=1}}\int_{|\theta-2\pi u/\mf{n}_D|\leq B/\sqrt{N}} G(re^{i\theta})e^{-iN\theta}\, d\theta
\notag \\
&=
\frac{1}{2\pi r^N}\sum_{\substack{u \bmod \mf{n}_D \\ (u,\mf{n}_D)=1}}{e^{\gamma_D(u,\mf{n}_D)}e(uN/\mf{n}_D)}\int_{-B/\sqrt{N}}^{B/\sqrt{N}} \exp\bigg(\frac{c_D}{\log(1/r)-i\theta}+\mathcal{O}_D(|1-r|)\bigg)e^{-iN\theta} \, d\theta,
\end{align}
where $\gamma_D(u,\mf{n}_D)$ is as defined in the statement of Proposition \ref{precise lem}. Since $1-r \asymp_{D} 1/\sqrt{N}$, we have $\exp(1-r) = 1 +\mathcal{O}_D(1/\sqrt{N})$. Therefore, the contribution from the error term in \eqref{Eq: Remaining Integral} is
\begin{align}\label{worst error}
&\frac{1}{2\pi r^N}\sum_{\substack{u \bmod \mf{n}_D \\ (u,\mf{n}_D)=1}}{e^{\gamma_D(u,\mf{n}_D)}e(uN/\mf{n}_D)}\int_{-B/\sqrt{N}}^{B/\sqrt{N}} \exp\bigg(\frac{c_D}{\log(1/r)-i\theta}\bigg) \cdot \mathcal{O}_D\bigg(\frac{1}{\sqrt{N}}\bigg) \cdot e^{-iN\theta} \, d\theta \notag \\
&\ll_D \frac{1}{\sqrt{N}}\exp\bigg(\frac{c_D}{\log(1/r)}-N\log r\bigg)\int_{-B/\sqrt{N}}^{B/\sqrt{N}} \, d\theta
\ll_D
\frac{1}{N}\exp\big(2\sqrt{c_DN}\big).
\end{align}
Here we have noted that the real part of $(\log(1/r)-i\theta)^{-1}$ is maximized when $\theta=0$, and the sum over $u \bmod \mf{n}_D$ which involves $\gamma_D(u,\mf{n}_D)$ is $\ll_D 1$ because none of the periodic zeta functions arising from \eqref{Eq: Gamma Def} contain fractions $au/v$ which lie in $\mathbb{Z}$. Therefore, we arrive at the following relation: 
\begin{align}\label{Eq: Fourier Coefficient Formula}
a_D(N)
=
\frac{\mf{S}^{0}(N)}{2\pi r^N}\int_{-B/\sqrt{N}}^{B/\sqrt{N}} e^{if(\theta)} \, d\theta
+
\mathcal{O}_D\bigg(\frac{1}{N}\exp\big(2\sqrt{c_DN}\big)\bigg)
\end{align}
where
\begin{align}\label{Eq: ftheta def}
f(\theta)=\frac{c_D}{ i(\log(1/r)-i\theta)}-N\theta
\end{align}
and 
\begin{align}\label{Eq: Sigma Def}
\mf{S}^{0}(N)= \sum_{\substack{u \bmod \mf{n}_D \\ (u,\mf{n}_D)=1}}{e^{\gamma_{D}(u,\mf{n}_D)}e(uN/\mf{n}_D)}.
\end{align}
Note here that $f(\theta)$ and $\mf{S}^{0}(N)$ both depend on $D$ as well, but we omit it from the notation for brevity.
\subsection{The singular integral}
We now focus on the integral in \eqref{Eq: Fourier Coefficient Formula}, localizing it further. By Lemma \ref{local lem},
\begin{align*}
\frac{1}{2 \pi r^N}\int_{\sqrt{c_D/N}\leq |\theta|\leq B/\sqrt{N}}e^{if(\theta)} \, d\theta
\ll_D N^{-\frac12}\exp\Big(\frac{3}{2}\sqrt{c_DN}\Big),
\end{align*}
which is exponentially smaller than our target bound $\exp(2\sqrt{c_DN})$. Also, we have
\begin{align*}
\frac{1}{2 \pi r^N}&\int_{N^{-3/4}\log N \leq |\theta|\leq \sqrt{c_D/N}}e^{if(\theta)} \, d\theta \\
&\ll_D
\exp\big(2\sqrt{c_DN}\big)
\int_{N^{-3/4}\log N}^\infty e^{-cN^{3/2}\theta^2}\, d\theta \ll_{D,J}
N^{-J} \exp\big(2\sqrt{c_DN}\big),
\end{align*}
for any $J>0$, and therefore this contribution is smaller than our target bound by any arbitrary power of $N$.

The remaining integral is given by 
\begin{align}\label{Eq: Integral}
\frac{1}{2\pi r^N}\int_{-N^{-3/4}\log N}^{N^{-3/4}\log N} e^{if(\theta)} \, d\theta.
\end{align}
In preparation for a Taylor expansion, we note that
\begin{align*}
if(0)&=\frac{c_D}{\log(1/r)},\qquad
f'(0)=  \frac{c_D}{\log^2(1/r)}-N=0, \quad f''(0)= \frac{2ic_D}{\log^3(1/r)} = 2i\frac{N^{3/2}}{\sqrt{c_D}}
\end{align*}
and 
\begin{align*}
f^{(3)}(0)&= -\frac{6c_D}{(\log(1/r))^4} = -\frac{6N^2}{c_D}\ll_D N^2, \quad f^{(4)}(\theta)= -\frac{24 i c_D}{(\log(1/r)-i\theta)^5}\ll_D N^{\frac52},
\end{align*}
for $\theta$ in the range of integration. Hence the integral \eqref{Eq: Integral} is 
\begin{align}\label{Eq: Truncated Gaussian}
\frac{e^{if(0)}}{2\pi r^N}
&\int_{-N^{-3/4}\log N}^{N^{-3/4}\log N}
e^{i\tfrac12 f''(0)\theta^2+i\tfrac16 f^{(3)}(0)\theta^3+\mathcal{O}_D(N^{5/2}\theta^4)} \, d\theta
\notag \\
&=
\frac{1}{2\pi}\exp\big(2\sqrt{c_DN}\big)
\int_{-N^{-3/4}\log N}^{N^{-3/4}\log N}e^{-N^{3/2}\theta^2/\sqrt{c_D}} \notag \\
&\quad \quad \quad\quad\quad\quad\quad\quad\quad \times \bigg (1-i\frac{N^2}{c_D} \theta^3+\mathcal{O}_D(N^{-\frac12}\log^6N)\bigg) \bigg(1+\mathcal{O}_D(N^{-\frac12}\log^4N)\bigg)\, d\theta
\notag \\
&=
\frac{c_D^{1/4}}{2\sqrt{\pi}N^{3/4}}\exp\big(2\sqrt{c_DN}\big)(1+\mathcal{O}_D(N^{-\frac12}\log^6N))
\end{align}
because the integral involving the term $\theta^3$ vanishes due to symmetry. Since the worst error we have so far other than \eqref{Eq: Truncated Gaussian} is from \eqref{worst error} and is of the form $N^{-1}\exp(2\sqrt{c_DN})$, in total then we have 
\begin{equation}\label{a}
a_D(N)
=
\frac{c_D^{1/4}\mf{S}^{0}(N)}{2\sqrt{\pi}N^{3/4}}\exp\big(2\sqrt{c_DN}\big)(1+\mathcal{O}_D(N^{-\frac14})).
\end{equation}
The constant $c_D^{1/4}/2\sqrt{\pi}$ is contained in the expression for $\alpha_D$ in Theorem \ref{gen thm}. Thus it suffices to show that the remaining constants arise from the singular series $\mf{S}^{0}(N)$.

\subsection{The singular series} The asymptotic \eqref{Eq: Key Asymptotic} in Theorem \ref{gen thm} will follow from the following lemma. 
\begin{comment}
In the following subsection, we then examine the singular series in more detail for the cases $\n=3,5$ giving Proposition \ref{n3 prop}. 
\end{comment}

\begin{lem}\label{Lem: Singular Series}
Let $\mf{S}^{0}(N)$ be as defined in \eqref{Eq: Sigma Def}. We have 
\[
\mf{S}^{0}(N)
=
e^{\sqrt{D}L(1,\chi_D)}
\mf{S}_D(N)\]
where $\mf{S}_D(N)=(-1)^N$ when $\n=2,$ and  
\[
\mf{S}_D(N)
=
2\sum_{1\leq u\leq \tfrac{\n-1}{2}}\cos\Big(\frac{2\pi uN}{\n}-\beta_{D}(u,\mf{n}_D)\Big), \quad \quad \n\geq 3,
\]
with $\beta_D(u,\mf{n}_D)$ as defined in \eqref{Beta def}. 
\end{lem}
\begin{proof}
Recall that 
\begin{align}\label{Eq: Sigma0}
\mf{S}^{0}(N)= \sum_{\substack{u \bmod \mf{n}_D \\ (u,\mf{n}_D)=1}}{e^{\gamma_{D}(u,\mf{n}_D)}e(uN/\mf{n}_D)}
\end{align}
where by Proposition \ref{precise lem}, we have
 \begin{align}\label{Eq: Gamma Evaluation}
\gamma_D(u,\mf{n}_D)
&=
-\frac{\sqrt{D} L(1,\chi_D)}{2} \cdot \frac{1-\mf{n}_D}{\mf{n}_D}
\notag \\
&\quad -
\frac{1}{\phi(\mf{n}_D)}\sum_{\psi\bmod \mf{n}_D} \sum_{a=1}^{\mf{n}_D-1}\ol{\psi}(a)\Big[\sqrt{D}P(0,au/\mf{n}_D)L(1,\chi_D\psi)
+
P(1,au/\mf{n}_D)L(0,\chi_D\psi)\Big].
 \end{align} \\
\noindent \textbf{Contribution from the principal character.} Let us isolate the contribution from the principal character $\psi_0$. We use throughout that $\mf{n}_D$ must be prime. Since both $\chi_D$ and $\psi_0$ are even, we have $L(0,\chi_D\psi_0)=0$. Hence the second term in the sum in \eqref{Eq: Gamma Evaluation} disappears.
 For the first term, note that 
 \[
\sum_{a=1}^{\mf{n}_D-1}\psi_0(a)P(s,au/\mf{n}_D)=\sum_{n\geq 1}\frac{1}{n^s}\sum_{a=1}^{\mf{n}_D-1}\psi_0(a)e(aun/\mf{n}_D).
 \]
 Since $\mf{n}_D$ is prime, for those $n$ not divisible by $\mf{n}_D$ the above inner sum is $\sum_{a=1}^{\mf{n}_D-1}e(a/\mf{n}_D)=-1$ since $u$ is also coprime to $\mf{n}_D$. If $n$ is divisible by $\mf{n}_D$, this sum is $\phi(\mf{n}_D)$. Hence, we obtain
 \begin{align}\label{Eq: Principal Character Relation}
    \sum_{a=1}^{\mf{n}_D-1}\psi_0(a)P(s,au/\mf{n}_D)
 =
 -L(s,\psi_0)+\phi(\mf{n}_D)\mf{n}_D^{-s}\zeta(s).
 \end{align}
Thus, the contribution from the principal character to the sum on the right hand side of \eqref{Eq: Gamma Evaluation} is
\[
 -\frac{\sqrt{D}}{\phi(\mf{n}_D)}L(1,\chi_D\psi_0)(-L(0,\psi_0)+\phi(\mf{n}_D)\zeta(0))
 =
 \frac{\sqrt{D}}{2}L(1,\chi_D\psi_0)
 =
\frac{\sqrt{D}}{2}(1+\mf{n}_D^{-1})L(1,\chi_D)   
\] 
since $L(0,\psi_0)=0$, $\zeta(0)=-\frac12$ and $\chi_D(\mf{n}_D)=-1$. Putting together this contribution with the first term in \eqref{Eq: Gamma Evaluation}, we have
\begin{align*}
\gamma_D(u,\mf{n}_D) = \sqrt{D}L(1,\chi_D) 
+ \gamma_D'(u,\mf{n}_D)
\end{align*}
where
\begin{align*}
\gamma_D'(u,\mf{n}_D)=
-\frac{1}{\phi(\mf{n}_D)}\sum_{\substack{\psi \bmod \mf{n}_D \\ \psi \neq \psi_0}} \sum_{a=1}^{\mf{n}_D-1}\ol{\psi}(a)\Big[\sqrt{D}P(0,au/\mf{n}_D)L(1,\chi_D\psi)
+
P(1,au/\mf{n}_D)L(0,\chi_D\psi)\Big].
\end{align*}
It remains to compute the contribution from the non-principal characters. \\

\noindent \textbf{Contribution from the non-principal characters.} Consider $\psi \bmod \mf{n}_D$ with $\psi \neq \psi_0$. Note that $\psi$ is primitive since $\n$ is prime. Thus we have 
\begin{align}\label{Eq: Connecting P with Gauss Sums}
\sum_{a=1}^{\n-1}\ol{\psi}(a)P(s,au/\mf{n}_D)
=
\sum_{n\geq 1}\frac{1}{n^s}\sum_{a=1}^{\n-1}\ol{\psi}(a)e(aun/\n)
=
\psi(u)\tau(\ol{\psi})L(s,\psi).
\end{align}
Hence we obtain
\begin{align}\label{Eq: Gamma Evaluation 1st Step}
\gamma_D'(u,\mf{n}_D)=
-\frac{1}{\phi(\mf{n}_D)}\sum_{\substack{\psi \bmod \mf{n}_D \\ \psi \neq \psi_0}} \psi(u) \tau(\overline{\psi})\Big[\sqrt{D}L(0,\psi)L(1,\chi_D\psi)
+
L(1,\psi)L(0,\chi_D\psi)\Big].
\end{align}
Since $\chi_D(-1)=1$, the parity of $\chi_D\psi$ equals the parity of $\psi$. Since the value at $s=0$ of even Dirichlet $L$-functions is zero, the above sum can be restricted to odd characters. The functional equation for odd characters $\chi \bmod q$ (see \cite[Ch. 5]{IK2004}) reads
\begin{align}\label{General FE}
L(1,\chi)=\frac{\pi\tau(\chi)}{iq}L(0,\ol{\chi}).
\end{align}
Applying this to \eqref{Eq: Gamma Evaluation 1st Step} gives 
\begin{align}\label{Eq: Gamma Evaluation I}
\gamma_D'(u,\mf{n}_D)& =
\frac{\pi i}{\phi(\mf{n}_D)}\sum_{\substack{\psi \bmod \mf{n}_D \\ \psi \neq \psi_0, \,  \psi(-1)=-1}} \psi(u) \tau(\overline{\psi}) \notag \\
&\quad \times \Big[\sqrt{D} \cdot \frac{\tau(\chi_D\psi)}{\n D} \cdot L(0,\psi)L(0,\chi_D\overline{\psi})
+
\frac{\tau(\psi)}{\n} L(0,\overline{\psi})L(0,\chi_D\psi)\Big].
\end{align}
We now use the following fact about Gauss sums: for characters $\chi_1,\chi_2$ of moduli $q_1,q_2$ with $(q_1,q_2)=1$,
\begin{align}\label{Gauss Sum Factoring}
\tau(\chi_1\chi_2)=\chi_1(q_2)\chi_2(q_1)\tau(\chi_1)\tau(\chi_2),
\end{align}
(see \cite[Thm. 9.6]{MV}). Since $(\n,D)=1$, we may apply this to obtain
\begin{align*}
\tau(\chi_D\psi)=\chi_D(\n)\psi(D)\tau(\chi_D)\tau(\psi)=-\sqrt{D} \psi(D)\tau(\psi).
\end{align*}
Also since $\psi$ is odd and primitive, $\tau(\ol{\psi})=\psi(-1)\ol{\tau(\psi)}
=
-
\ol{\tau(\psi)}$, so that $\tau(\ol{\psi})\tau(\psi)=-\n$. Therefore, we arrive at
\begin{align}\label{Eq: Gamma Evaluation II}
\gamma_D'(u,\mf{n}_D)& =
-\frac{\pi i}{\phi(\mf{n}_D)}\sum_{\substack{\psi \bmod \mf{n}_D \\ \psi \neq \psi_0, \,  \psi(-1)=-1}} \psi(u) \Big[
L(0,\overline{\psi})L(0,\chi_D\psi)-\psi(D) L(0,\psi)L(0,\chi_D\overline{\psi})\Big].
\end{align}

We now isolate the real character and pair up the complex conjugate characters. Since $\n$ is prime, there is only one non-principal real character $\psi(\cdot)=(\tfrac{\cdot}{\n})$. This appears in the sum, that is, the character is odd, if and only if $\n\equiv 3 \bmod 4$ and in this case, we have 
\begin{align}\label{Eq: Quadratic Reciprocity}
   \psi(D)=\Big(\frac{D}{\n}\Big)=(-1)^{\frac{(D-1)(\n-1)}{4}}\Big(\frac{\n}{D}\Big)
=
-1, 
\end{align}
by quadratic reciprocity along with the facts that $D\equiv 1\bmod 4$ and $(\frac{\n}{D})=-1$. Hence, we obtain 
\begin{align}\label{Eq: Gamma Evaluation III}
\gamma_D'(u,\mf{n}_D)& =
-\frac{\pi i}{\phi(\mf{n}_D)}\sum_{\substack{\psi \bmod \mf{n}_D \\ \psi \neq \psi_0, \,  \psi(-1)=-1 \\ \psi \textrm{ complex}}} \psi(u) \Big[
L(0,\overline{\psi})L(0,\chi_D\psi)-\psi(D) L(0,\psi)L(0,\chi_D\overline{\psi})\Big] \notag \\
&\quad -\mathbbm{1}_{\n \equiv 3\bmod 4} \cdot \frac{2\pi i}{\phi(\mf{n}_D)}\chi_{\n}(u) L(0,\chi_{\n})L(0,\chi_D\chi_{\n}),
\end{align}
where $\chi_{\n}$ denotes the real non-principal character modulo $\n$. 

To finish our proof, we note the following. Since the remaining odd characters occur in complex conjugate pairs, we may pair these terms up to obtain
\[
\gamma_D'(u,\mf{n}_D) = i \beta_D(u,\mf{n}_D)
\]
with $\beta_D(u,\mf{n}_D)$ as defined in \eqref{Beta def}. Therefore, we have 
\[
\gamma_D(u,\mf{n}_D)
=
\sqrt{D}
L(1,\chi_D)
-i \beta_D(u,\mf{n}_D).
\]
Since $\psi$ and $\chi_{\n}$ are both odd, $\psi(\n-u)=-\psi(u)$ and $\chi_{\n}(\n-u) = -\chi_{\n}(u)$. Hence we get $
\beta_D(\n-u,\mf{n}_D)=-\beta_D(u,\mf{n}_D).$ It follows that $e(uN/\n)e^{-i \beta_D(u,\mf{n}_D)}$ and $e((\n-u)N/\n)e^{-i\beta_D(\n-u,\mf{n}_D)}=e(-uN/\n)e^{i\beta_D(u,\mf{n}_D)}$ are conjugates. Pairing up these terms in the sum for $\mf{S}^{0}(N)$ in \eqref{Eq: Sigma0} gives
\begin{align*}
\mf{S}^{0}(N)&=  e^{\sqrt{D}L(1,\chi_D)} \sum_{\substack{u \bmod \mf{n}_D \\ (u,\mf{n}_D)=1}}{e^{-i\beta_{D}(u,\mf{n}_D)}e(uN/\mf{n}_D)} \\
& =e^{\sqrt{D}L(1,\chi_D)} \sum_{1\leq u\leq \tfrac{\n-1}{2}} \bigg( {e^{-i\beta_{D}(u,\mf{n}_D)}e(uN/\mf{n}_D)}+{e^{i\beta_{D}(u,\mf{n}_D)}e(-uN/\mf{n}_D)}\bigg) \\
& = 2e^{\sqrt{D}L(1,\chi_D)}\sum_{1\leq u\leq \tfrac{\n-1}{2}}\cos\Big(\frac{2\pi uN}{\n}-\beta_{D}(u,\mf{n}_D)\Big),
\end{align*}
for $\n\geq 3$. If $\n=2$, we can only have the term $u=1$ in \eqref{Eq: Sigma0} and $\beta_D(1,2)=0$. So we obtain 
\[
\mf{S}^{0}(N)=e^{\sqrt{D}L(1,\chi_D)}e(N/2)=(-1)^N e^{\sqrt{D}L(1,\chi_D)}.
\]
This completes the proof.
\end{proof}
\begin{comment}
\begin{rem}\label{Rem: Twisted Partition}
In connection with the above discussion, we observe that each factor in \eqref{Exact Formula} can be written down explicitly, with the exception of the terms of the form $p_{\mathrm{ord}}(m_j, \zeta_{D}^{b_j})$. A typical such factor is of the form
\begin{align}\label{Typical Factor}
\sum_{\lambda \vdash k}  \exp \left ( \frac{2 \pi i b \, \ell (\lambda)}{D}\right),
\end{align}
for some quadratic non-residue $b \bmod D$. Thus the problem naturally reduces to understanding exponential sums involving the number of parts $\ell(\lambda)$. Earlier works of Erdős–-Lehner \cite{Erdos-Lehner} and Hwang \cite{Hwang} established limiting distribution results for $\ell(\lambda)$, including central limit theorem type phenomena. The relation \eqref{Exact Formula}, together with \eqref{Typical Factor}, suggests that Conjectures \ref{Conj: Asymptotic} and \ref{Conj: Sign Changes} may be connected to finer questions concerning the distribution of $\ell(\lambda)$ in arithmetic progressions. This would be in close analogy to the study of the partition function $p(n)$ in arithmetic progressions, as in the work of Ono \cite{Ken-Ono}.
\end{rem}
\end{comment}
\subsection{Proof of Theorem \ref{gen thm}} We first show that $a_{D}(N) \in \mathbb{Q}(\sqrt{D})$. Recall that $q = e^{\frac{2 \pi i z}{\sqrt{D}}}$ and $\zeta_{D} = e^{\frac{2\pi i}{D}}$. By interchanging the products in \eqref{Defn: etaD}, we have the relation
\begin{align}\label{Eq: Key Relation}
\eta_{D}(z) & = q^{-\frac{L(-1,{\chi_{D}})}{2}} \cdot \prod_{\substack{n=1 \\ \chi_D(n) = -1 }}^\infty (1-q^n)^{-2} \cdot \prod_{\substack{n=1}}^\infty (1-q^n) \cdot \prod_{\substack{n=1 \\ (n,D)>1 }}^\infty (1-q^n)^{-1} \notag \\
&\quad \times \prod_{a=1}^{D-1} \prod_{n=1}^{\infty}(1-\zeta_{D}^a \cdot  q^n)^{{\chi_{D}}(a)}.
\end{align}
Therefore, it suffices to study only the coefficients arising from the products over $a$ in \eqref{Eq: Key Relation}. We treat the case where the product runs over the quadratic residues modulo $D$; the argument for the quadratic non-residues is concomitant.
\par
Since $D \equiv 1 \bmod{4}$, we have $\mathbb{Q}(\sqrt{D}) \subseteq \mathbb{Q}(\zeta_{D})$. Upon expanding the product
\[
\prod_{\substack{a=1 \\ \chi_D(a)=1}}^{D}(1 - \zeta_{D}^a q^n),
\]
the resulting coefficients are elementary symmetric polynomials in $\zeta_{D}^a$. We consider the Galois group $\mathrm{Gal}(\mathbb{Q}(\zeta_{D})/\mathbb{Q})$, which has order $\phi(D)$. Let $H$ denote the subgroup defined by
\[
H = \left\{ \sigma \in \mathrm{Gal}(\mathbb{Q}(\zeta_{D})/\mathbb{Q}) : \sigma(\zeta_{D}) = \zeta_{D}^k \text{ for some } k \in (\mathbb{Z}/D\mathbb{Z})^* \text{ with } \left( \frac{k}{D} \right) = 1 \right\}.
\]
This subgroup $H$ has index $2$ and its fixed field is quadratic. Moreover, the Gauss sum $\tau(\chi_D) = \sum_{a=1}^{D} \left( \frac{a}{D} \right) \zeta_{D}^a$
is equal to $\sqrt{D}$ and is fixed by all elements of $H$. Therefore the fixed field of $H$ is $\mathbb{Q}(\sqrt{D})$. Since symmetric polynomials in $\zeta_{D}^a$ (with $a$ ranging over quadratic residues modulo $D$) are fixed under $H$, the coefficients lie in $\mathbb{Q}(\sqrt{D})$. Note that the expressions involve only roots of unity, so they are algebraic integers. Hence the coefficients lie in the ring of integers $\mathcal{O}_K$ of \( K=\mathbb{Q}(\sqrt{D}) \), which is \( \mathbb{Z}\left[\frac{1 + \sqrt{D}}{2}\right] \).
\par
Finally, the asymptotic relation \eqref{Eq: Key Asymptotic} follows immediately from \eqref{a} and Lemma \ref{Lem: Singular Series}. \qed 

\subsection{Proof of Proposition \ref{n3 prop}} Recall the definition of $\beta_D(u,\mf{n}_D)$ from \eqref{Beta def}. We need the following lemma, which allows us to write $\beta_D(u,\mf{n}_D)$ in terms of values of $L$-functions at $s=1$.
\begin{lem}\label{L1 lem}
Let $\beta_D(u,\mf{n}_D)$ as defined in \eqref{Beta def}. We have 
\begin{align*}
\beta_D(u,\mf{n}_D)& =
\frac{2 \mf{n}_D \sqrt{D}}{\pi \phi(\mf{n}_D)} \operatorname{Re}\sideset{}{'}\sum_{\substack{\psi \bmod \mf{n}_D \\ \psi \textnormal{ complex, odd}}} \psi(u) \Big[
 L(1,\overline{\psi})L(1,\chi_D\psi)-\psi(D)L(1,\psi)L(1,\chi_D\overline{\psi})\Big] \notag \\
&\quad + \mathbbm{1}_{\n \equiv 3\bmod 4} \cdot \frac{2\mf{n}_D \sqrt{D}}{ \pi \phi(\mf{n}_D)}\chi_{\n}(u) L(1,\chi_{\n})L(1,\chi_D\chi_{\n}).
\end{align*}
with the same convention for the primed sum as defined in \eqref{Beta def}.
\end{lem}
\begin{proof}
For an odd character $\chi \bmod q$, by \eqref{General FE} we have 
\[
L(0,\chi) = \frac{\tau(\chi)}{i\pi} L(1,\overline{\chi}).
\]
By \eqref{Gauss Sum Factoring}, we also have
\[
\tau(\ol{\psi})\tau(\chi_D\psi)
=
\tau(\ol{\psi})\chi_D(\n)\psi(D)\tau(\chi_D)\tau(\psi)
=
\n\sqrt{D}\psi(D).
\]
Therefore, we obtain
\begin{align*}
    L(0,\overline{\psi})L(0,\chi_D\psi) = -\frac{1}{\pi^2}L(1,\psi)L(1,\chi_D\overline{\psi}) \tau(\ol{\psi})\tau(\chi_D\psi) = -\frac{\mf{n}_D \sqrt{D}}{\pi^2} \psi(D)L(1,\psi)L(1,\chi_D\overline{\psi}).
\end{align*}
A similar treatment for the products $L(0,\psi)L(0,\chi_D\overline{\psi})$ \and $L(0,\chi_{\n})L(0,\chi_D\chi_{\n})$ gives us the desired result. For clarity, we mention that we have used \eqref{Eq: Quadratic Reciprocity} while evaluating $\chi_{\n}(D)$.
\end{proof}

%By the functional equation and formulae for Gauss sums we can just as easily write $\beta_u$ in terms of values at $s=0$ of the relevant Dirichlet $L$-functions. Indeed since for odd characters $\chi$ modulo $q$,
%\[
%L(1,\chi)=\frac{\pi\tau(\chi)}{iq}L(0,\ol{\chi}),
%\]
%we have
%\begin{align*}
%&
%\frac{\n\sqrt{D}}{\pi}\Big[L(1,\ol{\psi})L(1,\chi_D\psi)
%-
%\psi(D)L(1,\psi)L(1,\chi_D\ol{\psi})\Big]
%\\
%= &
%\frac{\n\sqrt{D}}{\pi}\Big[\frac{\pi^2\tau(\ol{\psi})\tau(\chi_D\psi)}{i^2D\n^2}L(0,\psi)L(0,\chi_D\ol{\psi})
%-
%\psi(D)\frac{\pi^2\tau({\psi})\tau(\chi_D\ol{\psi})}{i^2D\n^2}L(0,\ol{\psi})L(0,\chi_D{\psi})\Big]
%\\
%= &
%\pi\Big[L(0,\ol{\psi})L(0,\chi_D{\psi})-\psi(D)L(0,{\psi})L(0,\chi_D\ol{\psi})\Big]
%\end{align*}
%since 
%\[
%\tau(\ol{\psi})\tau(\chi_D\psi)
%=
%\tau(\ol{\psi})\chi_D(\n)\psi(D)\tau(\chi_D)\tau(\psi)
%=
%\n\sqrt{D}\psi(D).
%\]
%Thus we have 
%\begin{multline*}
%\beta_u
%=
%\frac{2\pi}{\phi(\mf{n}_D)}
%\operatorname{Re}\sideset{}{'}\sum_{\substack{\psi\!\!\!\mod \n\\\psi\,\,\mathrm{odd,\,\,complex}}}
%\psi(u)
%\Big[L(0,{\psi})L(0,\chi_D\ol{\psi})
%-
%\psi(D)L(0,\ol{\psi})L(0,\chi_D{\psi})\Big]
%\\
%+
%\mathds{1}_{\n \equiv 3\!\!\!\mod 4}\frac{2\pi}{\phi(\mf{n}_D)}\chi_{\n}(u)L(0,\chi_{\n})L(0,\chi_D\chi_{\n}).
%\end{multline*}

\par
We are now ready to prove Proposition \ref{n3 prop}.
\begin{proof} In the case $\n=3$, it is easier to use the formula for $\beta_D(u,\mf{n}_D)$ in terms of $L$-values at $s=1$ as in Lemma \ref{L1 lem}. The characters $\chi_3$ and $\chi_D \chi_3$ correspond to the imaginary quadratic fields $\mathbb{Q}(\sqrt{-3})$ and $\mathbb{Q}(\sqrt{-3D})$ respectively. We shall use the class number formula for imaginary quadratic fields which states
\begin{equation*}
L(1,\chi)=\frac{2\pi h(d)}{w \sqrt{\lvert d \rvert}}, \quad \chi = \left( \frac{d}{\cdot}\right), \quad d<0
\end{equation*} 
where $h(d)$ is the class number of $\mb{Q}(\sqrt{d})$ and $w$ is the number of roots of unity in this field. We have
\[
\beta_D(1,3)
= \beta =
\frac{1}{\pi}\cdot\sqrt{3}L(1,\chi_3)\cdot\sqrt{3D}L(1,\chi_D\chi_3)
=
\frac{1}{\pi}\cdot \frac{2\pi h(-3)}{6}\cdot \frac{2\pi h(-3D)}{2}
=
\frac{\pi h(-3D)}{3}.
\]
Since $D \equiv 1 \bmod 4$ is squarefree and $(3,D)=1$ since $\n=3$, $h(-3 D)$ is even. 
\begin{comment}
For d>0, if t denotes the number of distinct prime factors of d, then 2^{t-1} divides h(-d). \textcolor{red}{Since $D$ is squarefree, $3D$ is not a prime power and hence $h(3D)$ is even -- need to check this is true, seems so for small $D$}. 
\end{comment}
In particular, $h(-3D)\equiv 0,2,4 \bmod 6$ which shows that the pattern of the cosine factor takes the form 
\[
 \cos\bigg(\frac{2\pi}{3}N-\beta\bigg)=1,-\tfrac12,-\tfrac12.
\]
When $D=17$ or $41$, we have $h(-3D)=2$ and this sequence occurs exactly when $N \equiv 1,2,3\bmod 3$. 

The only other possible sign pattern is $-1,\frac12,\frac12$ which can only occur when $h(-3D)\equiv 1,3,5 \bmod 6$. But since $h(-3D)$ is even, this never happens.

When $\n=5$, the real character in $\beta_D(u,\mf{n}_D)$ has no contribution. We need to consider a pair of complex conjugate characters. Let $\psi$ be the character with the sequence of values $\psi(n)=1,i,-i,-1$ for $n$ increasing. Then we have
\begin{align}\label{Eq: Sigma n_D=5}
\mf{S}_D(N)
=
2\cos\bigg(\frac{2\pi }{5}N-\beta_D(1,5)\bigg)+2\cos\bigg(\frac{4\pi }{5}N-\beta_D(2,5)\bigg)
\end{align}
where 
\[
\beta_D(u,5)=\frac{\pi}{2}\operatorname{Re}\bigg(\psi(u) \Big[
L(0,\overline{\psi})L(0,\chi_D\psi)-\psi(D) L(0,\psi)L(0,\chi_D\overline{\psi})\Big]\bigg).
\]
By the formula 
\begin{align}\label{Eq: s=0}
L(0,\chi)=-\frac{1}{q}\sum_{a=1}^q \chi(a)a,
\end{align}
valid for $\chi \bmod q$, we find that 
\[
L(0,{\psi})=
\frac{3+i}{5}, \quad \textrm{and} \quad L(0,\overline{\psi})=
\frac{3-i}{5}.
\]
\par
Now, the first moduli $D\equiv 1 \bmod 4$ which has $\n=5$ is $D=73$. In this case, $\psi(D)=-i$ and a short computation using \eqref{Eq: s=0} gives 
\[
L(0,\chi_D\psi)
=
-2+4i, \quad \textrm{and} \quad L(0,\chi_D\overline{\psi})
=
-2-4i.
\]
Inputting these values, we find 
\[
\beta_{73}(1,5)=\frac{6\pi}{5},
\quad \textrm{and} \quad
\beta_{73}(2,5)=-\frac{6\pi}{5}.
\]
\begin{comment}
\[
\beta_{73}(u,5)
=
\frac{\pi}{2}\operatorname{Re}\bigg(\psi(u)
\Big[\frac{3-i}{5}
(-2+4i)+i
\frac{3+i}{5}
(-2-4i)\Big]\bigg)
\]
so that 
\[
\beta_{73}(1,5)=\frac{\pi}{2}(-\frac25+\frac{14}{5})=\frac{6\pi}{5},
\quad \textrm{and} \quad
\beta_{73}(2,5)=-\frac{6\pi}{5}.
\]
\end{comment}
Putting this in \eqref{Eq: Sigma n_D=5} followed by a numerical check yields the desired conclusion.
\end{proof}

\section{Proof of Theorem \ref{mod 5 thm}}\label{Sec: Sketch Proof}
Here we give a brief sketch of how one can refine the methods involved in the proof of Theorem \ref{gen thm} to obtain Theorem \ref{mod 5 thm}.
\par
We first note that with a little extra effort, the asymptotic given in Corollary \ref{main thm} can be improved to the series expansion
\[
a_5(N)
= \frac{(-1)^N}{N^{3/4}}
\frac{\phi^2}{2\sqrt{5}}\exp\Big(\frac{2\pi\sqrt{N}}{5}\Big)
\bigg(1+\sum_{j=1}^Jb_jN^{-j/2}+\mathcal{O}(N^{-(J+1)/2})\bigg)
\]
for some computable constants $b_j$ and $\phi = (1+\sqrt{5})/2$. Indeed, by shifting contours arbitrarily far to the left in the proof of Proposition \ref{precise lem}, we acquire an asymptotic expansion of $\log G(q)$ as $q\to -1$ in terms of powers of $\log(-1/q)$ to any required degree of accuracy. We combine this with an arbitrarily high precision Taylor expansion of $f(\theta)$ in the stationary phase argument. Expanding the resulting sub-constant terms in the exponential into a series and integrating them against the Gaussian then yields the desired result. In particular, we note that after Taylor expanding $f(\theta)$, the terms 
\[
\tfrac{1}{m!} f^{(m)}(0)\theta^m, \quad m \geq 3
\]
are $o(1)$. Therefore using the relation $e^x = 1+x+\cdots$, the contribution from $m$ odd shall vanish due to the symmetry of the integral.

For the secondary terms, we note that from Proposition \ref{precise lem}, the next largest constant in the leading coefficient of the expansion of $\log G(q)$ occurs around the points $q\to \zeta_5^2,\zeta_5^3$. This is the regime $D \mid v$. Therefore using \eqref{log F 3}, we have 
\[
\log G(re^{i\theta+2\pi i u/5})
=
\frac{4\pi^2}{125(\log(1/r)-i\theta)}
+a_5(u,5)\log(\log(1/r)-i\theta)+b_5(u,5)+\mathcal{O}(1-r),
\]
for $u=2,3$ and computable constants $a_5(u,5)$ and $b_5(u,5)$. Again, the lower order terms here are computable to any required degree of accuracy by shifting contours further to the left in Proposition \ref{precise lem}. 

From a more detailed analysis of the arguments of Proposition \ref{precise lem}, we find that 
\[
a_5(u,5)=0, \quad u=2,3.
\]
To see this, first note that we are in the case $D \mid v$; in fact, we have $v=D=5$. There is no contribution from $S_{21}$, and we only need to evaluate the contribution from $S_{22}$. The coefficient of $\log(\log(1/q))$ arises from the contribution of the double pole at $s=0$ in \eqref{Eq: Double Pole}. This only happens when $\psi = \chi_5$ in which case, $L(1+s,\chi_5\psi) = L(1+s,\psi_0)$ where $\psi_0$ is the principal character modulo $5$. A quick computation shows that the residue from the double pole of the integrand in \eqref{Eq: Double Pole} when $\psi = \chi_5$ is
\begin{align}\label{Eq: Residue at Double Pole}
\prod_{p \mid 5}& \bigg(1-\frac{1}{p} \bigg)  \cdot \bigg( P'(0,au/5)+P(0,au/5) \bigg( \sum_{p \mid 5} \frac{\log p}{p-1} -\log\log(1/q) \bigg) \bigg) \notag \\
&= \frac{4}{5} \bigg( P'(0,au/5)+P(0,au/5) \bigg( \frac{\log 5}{4} -\log\log(1/q) \bigg) \bigg).
\end{align}
Summing over $a$ in \eqref{Eq: Double Pole}, it suffices to check that
\begin{align}\label{Eq: Polylogarithm}
\sum_{a=1}^{4} \chi_D(a)\, P(0,au/5) = 0, \qquad u=2,3.
\end{align}
A revisit to \eqref{Eq: Connecting P with Gauss Sums} and the fact $L(0,\chi_5)=0$ since $\chi_5$ is even gives the desired claim.
\par
Then, in-keeping with the saddle point method, as before we choose the contour around the points $\zeta_5^2,\zeta_5^3$ to have radius such that the derivative of the function 
\[
f(\theta)=\frac{4\pi^2}{125 i(\log(1/r)-i\theta)}-N\theta
\] 
vanishes at $\theta=0$, namely $\log(1/r)=\sqrt{4\pi^2N/125}=2\pi\sqrt{N}/(5\sqrt{5})$. Then combining the contributions from both points $\zeta_5^2,\zeta_5^3$, the methods of stationary phase give a leading order contribution of 
\begin{align*}
&\bigg( \sum_{u=2,3} e^{b_5(u,5)-2\pi iuN/5} \bigg)\cdot \frac{1}{2\pi r^N}e^{if(0)}\int_\mb{R}e^{\tfrac{i}{2}f''(0)\theta^2} \, d\theta \\
&=
\frac{(4\pi^2/125)^{1/4}}{2\sqrt{\pi}N^{3/4}}\bigg(\sum_{u=2,3}e^{b_5(u,5)-2\pi i uN/5}\bigg)\exp\bigg(\frac{4\pi\sqrt{N}}{5\sqrt{5}}\bigg),
\end{align*}
since 
\[
if(0)=\frac{2\pi}{5\sqrt{5}}\sqrt{N},\quad \textrm{and} \quad f''(0)=2i\frac{N^{3/2}}{(4\pi^2/125)^{1/2}}.
\]
Therefore, all it remains is to show that
\begin{align}\label{Eq: Aim}
\sum_{u=2,3}e^{b_5(u,5)-2\pi i uN/5}=2\phi^2\cos\bigg(\frac{4\pi N}{5}+\frac{8\pi}{25}\bigg),
\end{align}
where $\phi = (1+\sqrt{5})/2.$ To prove \eqref{Eq: Aim}, we need to evaluate $b_5(u,5)$ for which we need a more detailed analysis of Proposition \ref{precise lem}.
\par
Firstly, let's treat the contribution from $S_1$ in Proposition \ref{precise lem} to $b_5(u,5)$. Our case is $D \mid v$, in fact, $v=D=5$. Therefore $v'=1$ and only $S_{11}'$ contributes. The contribution from $S_{11}'$ is evaluated as follows:
\begin{align*}
\sum_{a=1}^{4} \chi_5(a)P(1,au/5)\zeta(0,a/5) = \sum_{a=1,2} \chi_5(a) \zeta(0,a/5) (P(1,au/5)-P(1,-au/5))
\end{align*}
because $\chi_5(-a)=\chi_5(a)$ and $\zeta(0,(5-a)/5)=-\zeta(0,a/5)$ by using the relation $\zeta(0,\alpha) = \tfrac{1}{2}-\alpha$, $0<\alpha<1$. Now, we use the relation $P(1,\alpha) =\textrm{Li}_1(e^{2 \pi i \alpha}) = -\log(1-e^{2\pi i \alpha})$ where $\textrm{Li}_s(z)$ is the polylogarithm function. This yields
\begin{align*}
P(1,\alpha) - P(1,-\alpha) &= -\log(1-e^{2\pi i \alpha})+\log(1-e^{-2\pi i \alpha}) \\
&= -\log(1-e^{2\pi i \alpha})+\log(1-e^{2\pi i \alpha})+\log(-1)-\log(e^{2\pi i \alpha})\\
&=2 \pi i \bigg(\frac{1}{2}-\alpha\bigg), \quad 0<\alpha<1.
\end{align*} 
Hence we obtain
\begin{align*}
2 \pi i \sum_{a=1,2} \chi_5(a) \zeta(0,a/5) (P(1,au/5)-P(1,-au/5)) = 2 \pi i\sum_{a=1,2} \chi_5(a) \bigg(\frac{1}{2}-\frac{a}{5} \bigg)\bigg(\frac{1}{2} -\frac{au}{5} \bigg).
\end{align*}
Therefore, we record the following contributions from $S_1$:
\begin{align}\label{Contribution 1}
\frac{3\pi i}{25}, \quad u=2, \quad \textrm{and} \quad -\frac{3\pi i}{25}, \quad u =3.
\end{align}
\par
Next, we look at the contribution from $S_2$. This is given by the contributions to the integral \eqref{Eq: Double Pole}. For the double pole which happens when $\psi = \chi_5$, following \eqref{Eq: Residue at Double Pole} and \eqref{Eq: Polylogarithm},
the contribution is
\begin{align*}
    \frac{4}{5} \cdot \frac{1}{\phi(5)} \sum_{a=1}^{4} \chi_5(a)\, P'(0,au/5),
\end{align*}
because the $P(0,au/5)$ term vanishes after summing over $a$. We can again use the relation \eqref{Eq: Connecting P with Gauss Sums}. By analytic continuation of both sides (note that $au/5 \notin \mathbb{Z}$ and $\psi$ is non-principal), we can take derivatives with respect to $s$. Applying this here, the contribution becomes
\[
\frac{1}{5} \chi_5(u)\tau(\chi_5)L'(0,\chi_5).
\]
Since $\chi_5$ is even, an application of the functional equation shows that $L'(0,\chi_5) =\tfrac{\sqrt{5}}{2}L(1,\chi_5)$. Recall that $L(1,\chi_5) = \tfrac{2}{\sqrt{5}}\log \phi$. Hence we record the contribution from the double pole case:
\begin{align}\label{Contribution 3}
-\frac{1}{\sqrt{5}}\log \phi, \quad u=2,3.
\end{align}
\par
Finally, we again look at the contribution from $S_2$ but this time, the contributions are from the simple pole in the integral \eqref{Eq: Double Pole}. These occur when $\psi \neq \chi_5$. We separate into $\psi = \psi_0$ principal, and $\psi \neq \psi_0$.\\
When $\psi = \psi_0$, we obtain 
\begin{align*}
    \frac{1}{\phi(5)} L(1,\chi_5 \psi_0) \sum_{a=1}^{4} \psi_0(a)\, P(0,au/5)= \frac{L(1,\chi_5)}{\phi(5)} \bigg(-L(0,\psi_0) + \phi(5) \zeta(0)\bigg) =- \frac{1}{2} L(1,\chi_5),
\end{align*}
similar to how we evaluated in \eqref{Eq: Principal Character Relation}. Therefore, the contribution from the principal character $\psi_0$ is
\begin{align}\label{Eq: Contribution 4}
   - \frac{1}{2} L(1,\chi_5) = -\frac{1}{\sqrt{5}}\log \phi, \quad u=2,3. 
\end{align}
Our last contribution is from the complex characters. Let $\psi$ be the character with the sequence of values $\psi(n)=1,i,-i,-1$ for $n$ increasing. Using \eqref{Eq: Connecting P with Gauss Sums} and \eqref{General FE}, the contribution here is
\begin{align*}
\frac{1}{\phi(5)} &\bigg( L(1,\chi_5 \psi)\sum_{a=1}^{4} \overline{\psi}(a) \, P(0,au/5) + L(1,\chi_5 \overline{\psi})\sum_{a=1}^{4} \psi(a) \, P(0,au/5) \bigg) \\
&= \frac{1}{4} \bigg( L(1,\chi_5 \psi)\psi(u)\tau(\ol{\psi})L(0,\psi) + L(1,\chi_5 \overline{\psi})\ol{\psi}(u)\tau({\psi})L(0,\ol{\psi}) \bigg) \\
&= \frac{\pi}{20i} \bigg( \psi(u) \tau(\chi_5 \psi)\tau(\ol{\psi})L(0,\psi) L(0,\chi_5 \ol{\psi}) + \ol{\psi}(u)\tau(\chi_5\ol{\psi})\tau({\psi})L(0,\ol{\psi}) L(0,\chi_5 \psi) \bigg).
\end{align*}
Note that $\chi_5 \psi = \ol{\psi}$. Thus the above simplifies to
\begin{align*}
\frac{\pi}{20i} &\bigg( \psi(u)\tau(\ol{\psi})^2L(0,\psi)^2 + \ol{\psi}(u)\tau({\psi})^2L(0,\ol{\psi})^2\bigg)=\frac{\pi}{10i} \operatorname{Re}\bigg( \psi(u)\tau(\ol{\psi})^2L(0,\psi)^2  \bigg).
\end{align*}
Therefore it suffices to evaluate $\tau(\ol{\psi})^2$ and $L(0,\psi)^2$. We already calculated before in the proof of Proposition \ref{n3 prop} that $L(0,{\psi})=
\tfrac{3+i}{5}$. A quick evaluation of the Gauss sum gives $\tau(\ol{\psi})^2 = -\sqrt{5}(1-2i)$. So we obtain the following contributions:
\begin{align}\label{Eq: Contribution 5}
    \frac{\pi}{10i} \operatorname{Re}\bigg( -\frac{i}{5\sqrt{5}}\cdot (1-2i) \cdot (3+i)^2  \bigg) = \frac{\pi i}{5\sqrt{5}}, \quad u=2, \quad \textrm{and} \quad -\frac{\pi i}{5\sqrt{5}}, \quad u=3.
\end{align}
Combining all the contributions from \eqref{Contribution 1}, \eqref{Contribution 3}, \eqref{Eq: Contribution 4}, \eqref{Eq: Contribution 5}, we arrive at
\begin{align*}
    b_5(2,5) = 2\log \phi-\frac{8 \pi i }{25}, \quad \textrm{and} \quad
    b_5(3,5) = 2 \log \phi+\frac{8 \pi i}{25}.
\end{align*}
It follows that
\begin{align*}
\sum_{u=2,3}e^{b_5(u,5)-2\pi i uN/5} & = e^{2 \log \phi} \bigg( e^{-\frac{8 \pi i}{25}-\frac{4\pi i N}{5} } + e^{\frac{8 \pi i}{25}-\frac{6\pi i N}{5} }\bigg) \\
& = 2\phi^2 \operatorname{Re} (e^{-\frac{8 \pi i}{25}-\frac{4\pi i N}{5} })= 2\phi^2\cos\bigg(\frac{4\pi N}{5}+\frac{8\pi}{25}\bigg)
\end{align*}
which completes the proof of \eqref{Eq: Aim}. In conclusion, we obtain the desired secondary term
\[
\frac{\sqrt{2}\phi^2}{5^{3/4}N^{3/4}}\cos\bigg(\frac{4\pi N}{5}+\frac{8\pi}{25}\bigg)\exp\bigg(\frac{4\pi\sqrt{N}}{5\sqrt{5}}\bigg),
\]
again, with lower order terms calculable if desired. 
\par
To complete the proof, we note from Proposition \ref{precise lem} that the next largest exponential term arises when $v=3$, the second least quadratic non-residue modulo $5$. A quick calculation shows that this contribution, as well as all other lower order terms is
\[
\ll N^{-\frac34} \exp \bigg(2 \sqrt{\frac{\sqrt{5}L(2,\chi_5) N}{3^2}} \bigg) \ll N^{-\frac34} \exp \bigg(\frac{4 \pi}{15} \sqrt{N}\bigg),
\]
which completes the proof. \qed
\begin{rem}\label{Comparion Remark}
We provide a brief comparison between the cases $D=5$ and $D=13$. A short computation using Proposition \ref{precise lem} and the fact that $L(2,\chi_{13}) = 4\pi^2/(13\sqrt{13})$ shows that the first two leading exponential terms in the asymptotic expansion of $a_{13}(N)$ arises from the cases $v=\mf{n}_{13}=2$ and $v=13$. The precise exponential terms are
\[
\exp\bigg( \frac{2 \pi \sqrt{N}}{\sqrt{13}} \bigg) \quad \textrm{and} \quad  \exp\bigg( \frac{4 \pi \sqrt{N}}{13} \bigg).
\]
Comparing it with the asymptotic \eqref{eq:two-cusp-theorem}, we see that the secondary exponential term is much closer to the leading exponential term when $D=5$ compared to $D=13$.
\end{rem}
\section{Exact Identities between Fourier Coefficients and Partitions}\label{sec: Fourier Coefficients} In this section, we explore several interesting connections between the Fourier coefficients $a_D(N)$ of $\eta_{D}(z)$ and the theory of partitions. Our starting point is the relation \eqref{Eq: Key Relation}. Let $\theta \in \mathbb{C}$ with $\lvert \theta \rvert \leq 1$. We define the following three products:
\begin{align*}
F_{\textrm{ord}}(q,\theta)= \prod_{\substack{n=1}}^{\infty} \left (1-\theta \cdot q^n \right )^{-1}, \, \, 
F_{e}(q,\theta) = \prod_{\substack{n=1}}^{\infty} \left (1-\theta \cdot q^n \right ), \,\, \textrm{and} \, \,
F_{\textrm{NR}}(q,D) = \prod_{\substack{n=1 \\ \chi_D(n)=-1}}^{\infty} \left (1-q^n \right )^{-1}.
\end{align*}
\begin{comment}
    \textrm{and} \quad F_{\textrm{QR}}(w) &= \prod_{\substack{n=1 \\ n \equiv \square \bmod q}}^{\infty} \left (1-w^n \right )^{-1}.
\end{comment}
In the special case when $D$ is prime, \eqref{Eq: Key Relation} gives
\begin{align}\label{Rearranging our F II}
\eta_{D}(z) &=  q^{-\frac{L(-1,{\chi_{D}})}{2}} \cdot  F_{\textrm{NR}}(q,D)^2 \cdot  F_{e}(q,1) \cdot F_{\textrm{ord}}(q^D,1) \prod_{\substack{a=1 \\ \chi_D(a)=1}}^{D}  F_{e}(q,\zeta_{D}^{a}) \prod_{\substack{a=1 \\ \chi_D(a)=-1}}^{D} F_{\textrm{ord}}(q, \zeta_{D}^{a}).
\end{align}
By Euler's Pentagonal Number Theorem, we have for $\lvert q\rvert <1$,
\begin{align}\label{Pentagonal Number Theorem}
F_e(q,\theta) &= 1+\sum_{k=1}^{\infty}(-1)^k\left(\theta^{3 k-1} q^{k(3 k-1) / 2}+\theta^{3 k} q^{k(3k+1) / 2}\right) =\sum_{k \in \mathbb{Z}} p_e(k,\theta) \, q^{g(k)},
\end{align}
where
\[
g(k) = \frac{k(3k-1)}{2}, \quad k \in \mathbb{Z}, \quad \textrm{and} \quad 
p_e(k,\theta) =\begin{cases}
  (-1)^k \cdot \theta^{3k-1},  & \textrm{if } k >0 \\
  (-1)^{k} \cdot \theta^{-3k} , \quad & \textrm{if } k \leq 0.
\end{cases}
\]
Note that the sequence $\{g(k) \}$ forms the generalized pentagonal numbers. Next, we write $\lambda \vdash k$ to denote that $\lambda$ is a partition of $k$. Then we have
\begin{align}\label{Ordinary Partition}
F_{\textrm{ord}}(q,\theta) = \sum_{k=0}^{\infty} p_{\textrm{ord}}(k,\theta) \, q^{k},
\end{align}
where 
\[
p_{\textrm{ord}}(k,\theta) = \sum_{\lambda \vdash k}  \theta^{\ell(\lambda)},
\]
and $\ell(\lambda)$ denotes the number of parts of the partition $\lambda$. In particular, when $\theta = 1$, we obtain $p_{\textrm{ord}}(k,1) =p(k)$, the ordinary partition function. Finally, we also have
\begin{align}\label{QNR Partitions}
F_{\textrm{NR}}(q,D) = \sum_{k=0}^{\infty}p_{\textrm{NR}} (k,D) \, q^k, \quad \lvert q \rvert < 1,
\end{align}
where $p_{\textrm{NR}}(k,D)$ denotes the number of partitions of $k$ into parts which are all quadratic non-residues modulo $D$. Putting together \eqref{Rearranging our F II}--\eqref{QNR Partitions}, we can write 
\begin{align}\label{Exact Formula}
    a_{D}(N) = \sum_{\vec{k}} \sum_{\vec{\ell}} \sum_{\vec{m}} p_{\textrm{NR}} (k_1,D)  p_{\textrm{NR}} (k_2,D) \prod_{i=0}^{\phi(D)/2} p_e(\ell_i, \zeta_{D}^{a_i}) \prod_{j=0}^{\phi(D)/2} p_{\textrm{ord}}(m_j, \zeta_{D}^{b_j}).
\end{align}
Here we set $a_0=b_0=1$ and \( \{a_i\} \) and \( \{b_j\} \) denote the \( i \)-th quadratic residue and the \( j \)-th quadratic non-residue modulo \( D \), respectively, in the interval \( [1, D] \), where $1 \leq i,j \leq \phi(D)/2$. The sum in \eqref{Exact Formula} is taken subject to the following conditions:

\begin{align*}
\vec{k} &= (k_1, k_2), \,  0 \leq k_1,k_2 \leq N, \\
\vec{\ell} &= (\ell_0, \ell_1, \cdots, \ell_{\frac{\phi(D)}{2}}), \, \ell_i \in \mathbb{Z}, \, 0 \leq g(\ell_i) \leq N, \\
\vec{m} &= (m_0, m_1, \cdots, m_{\frac{\phi(D)}{2}}), \, 0 \leq Dm_0, m_j \leq N, \textrm{ for all $j\geq 1$ }, \\
N &= k_1+k_2+\sum_{i=0}^{\phi(D)/2}g(\ell_i) \, +Dm_0 + \sum_{j=1}^{\phi(D)/2}m_j.
\end{align*}
In particular, the relation \eqref{Exact Formula} provides an exact expression for the Fourier coefficients $a_D(N)$ in terms of certain classical functions related to partitions. By Rademacher's work \cite{Rademacher1, Rademacher2}, we have
\begin{align}\label{Partition Exact Formula I}
p_{\textrm{ord}}(k,1) = p(k)=\frac{2 \pi}{(24 k-1)^{\frac{3}{4}}} \sum_{m=1}^{\infty} \frac{A_m(k)}{m} \cdot I_{\frac{3}{2}}\left(\frac{\pi}{6 m} \sqrt{24 k-1}\right),
\end{align}
where $A_m(k)$ is a certain Kloosterman sum and $I_{\nu}$ is the modified Bessel function of first kind. There are many results in the literature regarding exact formulas for $p_{\textrm{NR}} (k,D)$. When $D=5$, Lehner \cite{Lehner} proved that
\begin{align}\label{Partition Exact Formula II}
p_{\textrm{NR}} (k,5) &= \frac{2 \pi}{(60 k+11)^{\frac{1}{2}}} \sum_{\substack{m>0 \\
5 \nmid m}} \frac{\widetilde{A}_m(k)}{m} I_1\bigg(\frac{\pi(60 k+11)^{\frac{1}{2}}}{15 m}\bigg) \notag \\
& \quad+\frac{\pi}{(60 k+11)^{\frac{1}{2}}} \sum_{\substack{m>0 \\
5 \nmid m}}\left|\csc \frac{\pi a m}{5}\right| \frac{\widetilde{B}_m(k)}{m} I_1\bigg(\frac{\pi(60 k+11)^{\frac{1}{2}}}{15 m}\bigg)
\end{align}
where $\widetilde{A}_m(k)$ and $\widetilde{B}_m(k)$ are certain Dedekind sums and $I_1$ is the modified Bessel function of the first kind. It would be interesting to obtain similar exact formulas for $a_D(N)$ in terms of Kloosterman sums and Bessel functions. 
\section{Table of Fourier Coefficients}\label{sec: Graphical Data}
\begin{comment}
\textcolor{red}{Some more graphs may be added.} We now present a graph illustrating the growth of the Fourier coefficients $a_{D}(N)$ for $D = 5,13$ and $17$.

\begin{figure}[H]
 \centering
 \includegraphics[width=0.6\linewidth]{LogComparison.png}
 \caption{Comparison between the growths of $\log \, \lvert a_{D}(N) \rvert $ for $D=5$ (in red), $D=13$ (in green) and $D=17$ (in orange), for $N =1$ to $100$.}
 \label{fig:D=5}
\end{figure}
\end{comment}

The following table contains the precise values of $a_{D}(N)$ for $D = 5,13$ and $17$.

\begin{table}[ht]
\centering
\renewcommand{\arraystretch}{1.2} % Increase vertical spacing
\begin{tabular}{|c|c|c|c|}
\hline
\textbf{\( N \)} 
& \textbf{$a_{5}(N)$} 
& \textbf{$a_{13}(N)$} 
& \textbf{$a_{17}(N)$} \\
\hline
1  & \( -1 - \sqrt{5} \) & \( -1 - \sqrt{13} \) & \( -1 - \sqrt{17} \) \\
\hline
2  & \( \frac{1}{2}(7 + \sqrt{5}) \) & \( \frac{1}{2}(15 + \sqrt{13}) \) & \( \frac{1}{2}(15 - \sqrt{17}) \) \\
\hline
3  & \( -2\sqrt{5} \) & \( -2 - 4\sqrt{13} \) & \( 19 - \sqrt{17} \) \\
\hline
4  & \( 4 - \sqrt{5} \) & \( 27 + \sqrt{13} \) & \( \frac{1}{2}(13 - 23\sqrt{17}) \) \\
\hline
5  & \( 6 - 2\sqrt{5} \) & \( -12\sqrt{13} \) & \( 69 - 11\sqrt{17} \) \\
\hline
6  & \( \frac{1}{2}(13 - 7\sqrt{5}) \) & \( 66 - 3\sqrt{13} \) & \( 139 - 23\sqrt{17} \) \\
\hline
7  & \( 11 - 5\sqrt{5} \) & \( 27 - 29\sqrt{13} \) & \( 166 - 70\sqrt{17} \) \\
\hline
8  & \( \frac{1}{2}(35 - 13\sqrt{5}) \) & \( 155 - 18\sqrt{13} \) & \( 492 - 89\sqrt{17} \) \\
\hline
9  & \( 19 - 11\sqrt{5} \) & \( 128 - 64\sqrt{13} \) & \( 818 - 182\sqrt{17} \) \\
\hline
10 & \( \frac{1}{2}(69 - 25\sqrt{5}) \) & \( \frac{1}{2}(715 - 119\sqrt{13}) \) & \( 1242 - 378\sqrt{17} \) \\
\hline
11 & \( 37 - 21\sqrt{5} \) & \( 364 - 148\sqrt{13} \) & \( 2567 - 561\sqrt{17} \) \\
\hline
12 & \( \frac{1}{2}(129 - 45\sqrt{5}) \) & \( 795 - 164\sqrt{13} \) & \( 4235 - 998\sqrt{17} \) \\
\hline
13 & \( 70 - 38\sqrt{5} \) & \( 912 - 332\sqrt{13} \) & \( 6780 - 1756\sqrt{17} \) \\
\hline
14 & \( 108 - 43\sqrt{5} \) & \( 1752 - 393\sqrt{13} \) & \( \frac{1}{2}(23637 - 5515\sqrt{17}) \) \\
\hline
15 & \( 134 - 62\sqrt{5} \) & \( 2160 - 706\sqrt{13} \) & \( 18977 - 4559\sqrt{17} \) \\
\hline
16 & \( 176 - 80\sqrt{5} \) & \( 3699 - 891\sqrt{13} \) & \( \frac{1}{2}(59823 - 14961\sqrt{17}) \) \\
\hline
17 & \( 233 - 103\sqrt{5} \) & \( 4761 - 1467\sqrt{13} \) & \( 48557 - 11627\sqrt{17} \) \\
\hline
18 & \( \frac{1}{2}(603 - 267\sqrt{5}) \) & \( \frac{1}{2}(15069 - 3855\sqrt{13}) \) & \( 76106 - 18308\sqrt{17} \) \\
\hline
19 & \( 377 - 175\sqrt{5} \) & \( 9986 - 2970\sqrt{13} \) & \( 117103 - 28745\sqrt{17} \) \\
\hline
20 & \( \frac{1}{2}(1017 - 429\sqrt{5}) \) & \( 15069 - 3957\sqrt{13} \) & \( \frac{1}{2}(363839 - 87715\sqrt{17}) \) \\
\hline
21 & \( 608 - 288\sqrt{5} \) & \( 20174 - 5854\sqrt{13} \) & \( 276958 - 67034\sqrt{17} \) \\
\hline
22 & \( \frac{1}{2}(1625 - 693\sqrt{5}) \) & \( \frac{1}{2}(58843 - 15677\sqrt{13}) \) & \( 417234 - 101814\sqrt{17} \) \\
\hline
23 & \( 985 - 455\sqrt{5} \) & \( 39390 - 11286\sqrt{13} \) & \( 629047 - 152045\sqrt{17} \) \\
\hline
24 & \( 1265 - 556\sqrt{5} \) & \( 56002 - 15115\sqrt{13} \) & \( \frac{1}{2}(1871277 - 453479\sqrt{17}) \) \\
\hline
25 & \( 1564 - 706\sqrt{5} \) & \( 74911 - 21265\sqrt{13} \) & \( 1381414 - 335786\sqrt{17} \) \\
\hline
\end{tabular}
\vspace{0.2cm}
\caption{$a_{D}(N)$ corresponding to $D=5,13$ and $17$, for \( N = 1 \) to \( 25 \).}

\label{tab:coefficients-5-13-17}
\end{table}

\FloatBarrier
\section{Acknowledgements}
We thank Alex Kontorovich, Don Zagier and Wadim Zudilin for valuable discussions and helpful comments on the subject of the paper.
\printbibliography
\end{document}